\documentclass[a4paper,12pt]{article}

\usepackage[portrait, top=3cm, bottom=3cm, left=2.5cm, right=2.5cm]{geometry}

\usepackage{verbatim}
\usepackage{todonotes}

\usepackage[OT2, T1]{fontenc} 
\usepackage{amsmath,amsfonts,amssymb,amsxtra,mathrsfs,mathtools,tikz-cd}
\usepackage[new]{old-arrows} 

\DeclareRobustCommand\longtwoheadrightarrow{\relbar\joinrel\twoheadrightarrow}

\usepackage{graphicx, tensor}
\usepackage[all]{xy}

\usepackage{cite}
\usepackage[pagebackref=true]{hyperref}
\hypersetup{hidelinks}

\usepackage{enumerate}

\numberwithin{equation}{section}

\usepackage[thmmarks, amsmath, amsthm]{ntheorem}

\theoremstyle{plain}
\newtheorem{thmalpha}{Theorem}

\newtheorem{thm}{Theorem}[section]
\newtheorem{lemm}[thm]{Lemma}
\newtheorem{prop}[thm]{Proposition}
\newtheorem{coro}[thm]{Corollary}
\newtheorem*{cor*}{Corollary}

\theoremstyle{definition}
\newtheorem{defn}[thm]{Definition}
\newtheorem{remk}[thm]{Remark}
\newtheorem{exam}[thm]{Example}

\newcommand{\Br}{\operatorname{Br}}
\renewcommand{\div}{\operatorname{div}}
\newcommand{\Div}{\operatorname{Div}}
\newcommand{\Ext}{\operatorname{Ext}}

\newcommand{\Hom}{\operatorname{Hom}}
\newcommand{\id}{\operatorname{id}}
\newcommand{\Img}{\operatorname{Im}}
\newcommand{\KD}{\operatorname{KD}}
\newcommand{\Ker}{\operatorname{Ker}}
\newcommand{\mult}{\operatorname{mult}}
\newcommand{\ob}{\operatorname{ob}}
\newcommand{\Pic}{\operatorname{Pic}}

\newcommand{\qt}{\operatorname{qt}}
\newcommand{\Rep}{\operatorname{Rep}}
\newcommand{\Res}{\operatorname{Res}}
\newcommand{\RHom}{\mathbf{R}\!\operatorname{Hom}}
\newcommand{\Spec}{\operatorname{Spec}}
\newcommand{\tor}{\operatorname{tor}}
\newcommand{\tors}{\operatorname{tors}}
\newcommand{\UPic}{\operatorname{UPic}}
\newcommand{\XType}{\mathbf{XType}}

\DeclareSymbolFont{cyrletters}{OT2}{wncyr}{m}{n}
\DeclareMathSymbol{\Be}{\mathalpha}{cyrletters}{"42}
\DeclareMathSymbol{\Sha}{\mathalpha}{cyrletters}{"58}

\newcommand{\Abf}{\mathbf{A}}

\newcommand{\Gbf}{\mathbf{G}}
\newcommand{\Qbf}{\mathbf{Q}}
\newcommand{\Zbf}{\mathbf{Z}}

\newcommand{\Acal}{\mathcal{A}} 

\newcommand{\Dcal}{\mathcal{D}} 
\newcommand{\Lcal}{\mathcal{L}}

\renewcommand{\H}{\mathrm{H}} 
\newcommand{\m}{\mathrm{m}} 

\newcommand{\lau}[1]{(\!(#1)\!)}

\title{Iterated toric descent obstruction and homogeneous spaces over geometric global fields}
\author{ Felipe Gambardella \\ \vspace*{-1ex} \small \textit{Université Paris-Saclay} \and Nguy\~{\^e}n M\d{a}nh Linh \\ \vspace*{-1ex} \small \textit{Université d'Artois}}
\date{}

\begin{document}

\maketitle

\begin{abstract}
    \noindent We study torsors under tori and groups of multiplicative type together with the obstruction they define to the existence of rational points in the setting of geometric global fields. Our methods give two new arithmetic applications: the iterated toric descent obstruction coincides with the toric decent obstruction and the toric obstruction is the only obstruction to the existence of rational points for homogeneous spaces of connected linear groups with either connected or abelian geometric stabilizers. These applications generalize some known results for classical global fields, due to Borovoi, Colliot-Thélène, and Sansuc, while exhibiting novel behaviors.
\end{abstract}

\tableofcontents

\section{Introduction}


Let $k$ be a number field and $X$ a variety over $k$. Understanding the set $X(k)$ of rational points is among the oldest objectives in number theory. Already determining whether the set $X(k)$ is empty is a daunting challenge for a general variety. A particularly fruitful approach is the local--global principle and the Brauer--Manin obstructions to the local--global principle introduced by Manin. Explicitly, a variety $X$ satisfies the local--global principle if the fact that $X$ admits points over completions at all places of $k$ implies that $X(k)$ is nonempty. The Brauer--Manin obstruction is incarnated by a cohomologically defined subset $X(\Abf_k)^{\operatorname{Br}}$ of the adelic points that contains the diagonal image of $X(k)$. We say that $X$ satisfies (BMLG) if either $X(\Abf_k)^{\operatorname{Br}}$ is empty or $X(k)$ is nonempty. \par

One case of particular interest where the Brauer--Manin obstruction sheds light on the behaviour of rational points is the case of homogeneous spaces. In \cite{Sansuc1981Arithmetique}, Sansuc proved that principal homogeneous spaces under linear connected $k$-groups satisfy (BMLG). Building on the work of Sansuc, Borovoi proved in \cite{Borovoi1996BMHomogeneouSpaces} that homogeneous spaces of linear connected $K$-groups with either connected or commutative geometric stabilisers satisfy~(BMLG). \par

Although the Brauer--Manin obstruction is a powerful tool to understand the local--global principle and the set of rational points, there are varieties that do not satisfy (BMLG) --- e.g. \cite{Sko19999BeyondBM} and \cite{Poonen2010Insufficiency}. This led to consideration of finer obstructions --- e.g. the descent obstruction in \cite{CTS1987Descente} or the étale-Brauer--Manin obstruction in \cite{Sko19999BeyondBM}. It turns out that these obstructions coincide. No obstruction is known to explain the failure of the local--global principle for all algebraic varieties \cite{Poonen2010Insufficiency}. A particularity of the descent obstruction is that it can be {\em iterated} as described below. A result that motivated part of this article is \cite[Theorem~1.2]{Cao2020IteratedDescent} which states that the iterated descent obstruction coincides with the descent obstruction. \par

In recent years, a local--global approach to rational points over other two-dimensional fields has been gaining interest --- e.g. \cite{CTGP2004ArithmeticLinear}, \cite{CTH2015DualCt}, \cite{Izquierdo2019Dualite}, \cite{ILA2021LocalGlobal}. In the present text, we are interested by \textit{geometric global fields} --- i.e. a finite extension of either $\mathbf{C}(\!(t)\!)(x)$ or $\mathbf{C}(\!(x,y)\!)$. In Section~\ref{subsection generalities fields} we explain that there is a natural set of places coming from the geometry of the fields and an obstruction also denominated Brauer--Manin obstruction. This setting was originally introduced in \cite{CTH2015DualCt} and \cite{Izquierdo2019Dualite}. It was proven \textit{loc. cit.} that the results of Sansuc generalize to this setting --- i.e. principal homogeneous spaces of connected linear groups over geometric global fields satisfy (BMLG). However, we can find in \cite[Section~4]{ILA2021LocalGlobal} an example of a homogeneous space of $\operatorname{SL}_n$ with toric stabiliser that does not satisfy (BMLG) showing that the results of Borovoi do not generalize to this setting. In \cite{ILA2021LocalGlobal}, Izquierdo and Lucchini Arteche introduced a refined obstruction $X(\mathbf{A}_K)^{\qt, \Be}$ and proved that this obstruction together with an unnatural condition on a larger set of valuations than the one described below is the only obstruction to the local-global principle. We remove this unnatural assumption, obtaining  the following. 

\begin{thmalpha}\label{thmalpha homogeneous spaces}
    Let $K$ be a geometric global field and $X$ be a homogeneous space of a connected linear algebraic group $G$ with geometric stabiliser $\bar{H}$. Assume that $G^{\mathrm{ss}}$ is simply connected and $\bar{H}^{\mathrm{torf}}$ is commutative (see Section~\ref{section Homogeneous spaces} for the notation). Then, $X(K)$ in nonempty whenever $X(\Abf_K)^{\qt, \Be} \neq \varnothing$.
\end{thmalpha}

\noindent In particular, by a known lemma of Borovoi, we obtain the following corollary.
\begin{cor*}
    Let $K$ be a geometric global field and $X$ be a homogeneous space of a connected linear algebraic group $G$ with connected geometric stabiliser $\bar{H}$. Then, $X(K)$ in nonempty whenever $X(\Abf_K)^{\qt, \Be} \neq \varnothing$.
\end{cor*}

\noindent The need to consider obstructions finer than Brauer--Manin to control the local--global principle over homogeneous spaces invites us to study the behaviour of further obstructions over geometric global fields. As explained in \cite[Section~6]{ILA2021LocalGlobal}, the obstruction $X(\Abf_K)^{\qt, \Be}$ compares with a more classical obstruction denominated the \textit{toric descent obstruction} defined as
\[
    X(\Abf_K)^{\tor} \coloneq \bigcap_{f \colon Y \xrightarrow{T} X} \bigcup_{[\sigma] \in \H^1(K,T)} f^{\sigma}(Y^{\sigma}(\Abf_K)),
\]
where the intersection runs over every torsor under an algebraic $K$-torus and $f^{\sigma} \colon Y^{\sigma}\to X$ denotes the Galois twist of $f$. A similar obstruction $X(\Abf_K)^{\mult}$ can be defined with respect to $K$-groups of multiplicative type. The second main result of this article is a description of the toric descent obstruction in terms of $X(\Abf_K)^{\qt, \Br_1}$ and the fact that the iterated descent obstruction coincides with toric descent obstruction (see Section~\ref{subsection generalities obstruction sets} for the definition).

\begin{thmalpha}\label{thmalpha tor,tor equals tor}
    Let $X$ be a smooth geometrically integral variety such that the group $\Pic(\bar{X})$ is finitely generated. Then, we have
    \[
        X(\Abf_K)^{\tor,\tor} = X(\Abf_K)^{\tor} = X(\Abf_K)^{\mult} = X(\Abf_K)^{\qt, \Br_1}.
    \]
\end{thmalpha}

\section{Preliminaries}


\subsection{Generalities on geometric global fields}\label{subsection generalities fields}

\begin{defn}
    A field $K$ is said to be a {\em geometric global field} if $K$ belongs to one of the following types:
    \begin{itemize}
        \item[(a)] a function field of a smooth projective curve $C$ geometrically integral over $\mathbf{C}\lau{t}$,~or
        \item[(b)] a finite extension of the field $\mathbf{C}\lau{x,y}$.
    \end{itemize}
\end{defn}

\noindent In the sequel, we consider local--global principles for geometric global fields with respect to a natural set of places coming from a natural Dedekind scheme associated to the field. Explicitly, the set of places under consideration is denoted by $C^{(1)}$ and defined depending on the type:
\begin{itemize}
    \item[(a)] it is the set of places coming from closed points of the smooth projective geometrically integral curve $C$ over $\mathbf{C}\lau{t}$, and
    \item[(b)] let $R$ be the integral closure of $\mathbf{C}[\![x,y]\!]$ in $K$. Note that $R$ is a local ring with maximal ideal $\mathfrak{m}$. Then, $C^{(1)}$ is the set of places coming from closed points of the (non-affine) Dedekind scheme $C \coloneq \operatorname{Spec}(R) \setminus \{\mathfrak{m}\}$.
\end{itemize}

\noindent In either case, we refer to $C$ as the \textit{Dedekind scheme associated to $K$}. \par

For every $v \in C^{(1)}$ we denote by $K_v$ the completion of $K$ at $v$ and by $\mathcal{O}_v$ the ring of integers of $K_v$. Let $X$ be a $K$-variety. We denote by $\bar{X}$ the base change of $X$ to the algebraic closure and for every $v \in C^{(1)}$ we denote by $X_v$ the base change of $X$ to $K_v$. We can define the set of adelic points of a variety $X$ by analogy with the case of classical global fields: there exists an open set $U \subseteq C$ such that $X$ admits a model $\mathcal{X}$ over $U$, the set of adelic points is defined as
\[
    X(\Abf_K) \coloneq \left\{(x_v)_{v \in C^{(1)}} \in \prod_{v \in C^{(1)}} X(K_v) \mid  x_v \in \mathcal{X}(\mathcal{O}_v) \text{ for almost all }v\right\}.
\]

\noindent We often write $x_{\mathbf{A}}$ to denote an element of $X(\Abf_K) $ whose local component at $v \in C^{(1)}$ is given by $x_v$. \par

Let $\mathcal{F}$ be a class of $K$ varieties. We say that $\mathcal{F}$ satisfies the local--global principle if for all $X \in \mathcal{F}$ we have
\[
    X(\Abf_K) \neq \varnothing \Longrightarrow X(K) \neq \varnothing.
\]
Similarly to the case of number fields, there exists a pairing
\begin{equation}\label{eq Brauer--Manin}
    \mathrm{BM}\colon X(\Abf_K) \times \Br (X) \longrightarrow  \Qbf / \Zbf    
\end{equation}
such that the diagonal image of $X(K)$ in $X(\Abf)$ is contained in the left kernel $X(\Abf)^{\Br(X)}$ of $\mathrm{BM}$.~Define 
\begin{align*}
    \Br_0 (X) & \coloneq \Img(\Br (K) \longrightarrow   \Br (X)) \\
    \Br_1 (X) & \coloneq \Ker (\Br (X) \longrightarrow   \Br (\bar{X}))\\
    \Br_{\mathrm{al}} (X) & \coloneq \Br_1 (X) / \Br_0 (X).
\end{align*}
It follows from the exact sequences \cite[Proposition~2.1]{CTH2015DualCt} and \cite[Théorème~1.6]{Izquierdo2019Dualite} that the pairing \eqref{eq Brauer--Manin} factorises through $\Br (X)/\Br_0(X)$. For every subgroup $B$ of $\Br(X)$ we define $X(\Abf_K)^{B}$ to be the set of adelic points orthogonal to every $\alpha \in B$ with respect to the pairing \eqref{eq Brauer--Manin}.
\begin{defn}
    Let $X$ be a variety and $B$ a subgroup of $\Br(X)$. We say that the Brauer--Manin obstruction with respect $B$ is the only obstruction to the local--global principle for $X$ if
    \[
    X(\Abf_K)^{B} \neq \varnothing \Longrightarrow X(K)\neq \varnothing.
    \]
\end{defn}

\noindent In this manuscript we mostly consider the group $B$ to be either $\Br_1(X)$ or 
\[
    \Be(X) \coloneq \Ker \left( \Br_1(X) \longrightarrow \prod_{v\in C^{(1)}} \Br_{\mathrm{al}} (X_v)\right).
\]


\subsection{Adelic spaces of tori}

Let $\Gamma$ be a profinite group. Recall that a $\Gamma$-\textit{lattice} is a finitely generated free abelian group endowed with a continuous action of $\Gamma$. We denote by $\Lcal_{\Gamma}$ the category whose objects are $\Gamma$-lattices and morphisms are $\Gamma$-equivariant morphisms of groups. We also consider the \textit{isogeny category} $\Lcal \otimes_{\Zbf} \Qbf$ whose objects are again $\Gamma$-lattices and the morphism set between two objects $L_1$ and $L_2$ is given by $\Hom_{\Lcal}(L_1,L_2) \otimes_{\Zbf}\Qbf$.
\begin{lemm} \label{lemm:CategoryOfToriUptoIsogeny}
    For any profinite group $\Gamma$, the isogeny category of $\Gamma$-lattices is a semisimple abelian category.
\end{lemm}

\begin{proof}
    Note that in contrast with $\Lcal_{\Gamma}$, the isogeny category $\Lcal \otimes_{\Zbf} \Qbf$ is an abelian category.
    There exists a fully faithful (additive) functor
    \[
        F \colon \Lcal_\Gamma \otimes_\Zbf \Qbf \longrightarrow \Rep_\Qbf(\Gamma), \qquad L \longmapsto L \otimes_\Zbf \Qbf,
    \]
    where $\Rep_\Qbf(\Gamma)$ denotes the category of finite-dimensional representations of $\Gamma$ over $\Qbf$. 
    \par
    
    If $V$ is such a representation, then the action of $\Gamma$ on $V$ factors through a finite quotient $G$. Choose a base for $V$ and let $L$ be its $\Zbf[G]$-span. Then $L$ is a $\Gamma$-lattice such that $L \otimes_\Zbf \Qbf \cong V$. Therefore, $F$ is essentially surjective, so it is an equivalence of abelian categories. Since $\Rep_\Qbf(\Gamma)$ is semisimple by Maschke's theorem, the lemma follows.
\end{proof}

\noindent Recalling that the functor that to an algebraic $K$-torus $T$ assigns its module of characters $\widehat{T}$ is a contravariant equivalence of categories, Lemma~\ref{lemm:CategoryOfToriUptoIsogeny} admits the following corollary.
\begin{coro} \label{coro:CategoryOfToriUptoIsogeny}
    Let $k$ be a field and $\pi: T \to S$ a surjective map of algebraic $k$-tori. Then, there exists a morphism $\sigma: S \to T$ and $n \in \mathbf{N}$ such that $\pi \circ \sigma = [n]_S$ where $[n]_S: S \to S$ is multiplication by $n$.
\end{coro}

\noindent Let $A$ be an abelian group. We fix the following notations: 
\begin{itemize}
    \item $A_{\div}$ denotes the maximal divisible subgroup of $A$,
    \item $A_{\infty-\div} \coloneq \bigcap_{n \in \mathbf{N}} n A$ is the group of infinitely divisible elements,
    \item $A_{\wedge}$ the limit $\varprojlim_{n\in \mathbf{N}} A/ nA$.
\end{itemize}

\begin{defn}
    Let $T$ be a $K$-torus. We define the \textit{adelic space} as 
    \[
    \Acal(T) \coloneq T(\Abf_K)/T(K).
    \]
\end{defn}

\noindent Throughout this text, we denote the group law of a $K$-torus $T$ multiplicatively while the group law of the adelic space $\mathcal{A}(T)$ is denoted additively.

\begin{coro} \label{coro:DivA(T)Surjective}
    Let $\pi \colon S \to T$ be a surjective morphism of $K$-tori. Then the induced~map 
    \[
        \pi \colon \Acal(S)_{\div} \longrightarrow \Acal(T)_{\div}
    \]
    is surjective.
\end{coro}
\begin{proof}
    By Corollary~\ref{coro:CategoryOfToriUptoIsogeny}, there exist $n \ge 1$ and a morphism $\sigma\colon T \to S$ such that
    \[
        \pi \circ \sigma = [n]_T.
    \]
    It follows that
    \[
        \Acal(T)_{\div} = n\Acal(T)_{\div} = \pi\bigl(\sigma(\Acal(T)_{\div})\bigr) \subseteq \pi(\Acal(S)_{\div}),
    \]
    which proves the result.
\end{proof}
For a $K$-torus $T$ we define the Tate--Shafarevich group as
\[
\Sha^n(K,T) \coloneq \Ker \left( \H^n(K,T) \to \prod_{v \in C^{(1)}} H^n(K_v,T) \right).
\]
The following proposition corresponds to \cite[Proposition~5.2, Théorème~5.2, and Corollaire~5.6]{HI2019Adelique} for fields of type (a). However, the result is new for fields of type (b).
\begin{prop}\label{prop properties adelic spaces}
    Let $K$ be a geometric global field and $T$ a $K$-torus. Then, we have
    \begin{enumerate}
        \item\label{prop properties adelic spaces: div = infity div} $\Acal(T)_{\div} = \Acal(T)_{\infty-\div}$.
        \item\label{prop properties adelic spaces: Poitou-Tate} We have an exact sequence
        \[
            T(K)_\wedge \longrightarrow T(\Abf_K)_\wedge \longrightarrow \Hom\bigl(\H^2(K,\widehat{T}), \Qbf/\Zbf\bigr) \longrightarrow \Hom\bigl(\Sha^2(K,\widehat{T}), \Qbf/\Zbf\bigr) \longrightarrow 0,
        \]
        where the second arrow is induced by the local cup-product pairings
        \[
            T(K_v) \times \H^2(K_v,\widehat{T}) \longrightarrow \Br(K_v) \cong \Qbf/\Zbf.
        \]
        \item\label{prop properties adelic spaces: finite exponent} The torsion subgroup of $\Acal(T)/\Acal(T)_{\div}$ has finite exponent.
    \end{enumerate}
\end{prop}
\noindent The proofs given in \cite{HI2019Adelique} can be followed word by word to obtain Proposition~\ref{prop properties adelic spaces} except for the following technical lemma.
\begin{lemm}\label{lemm: divisible subgroup cofinite quotient}
    There exists a divisible subgroup $U^0$ of $\Acal(\mathbf{G}_m)$ such that the torsion of the quotient $\Acal(\mathbf{G}_m)/ U^0$ is a group of cofinite type — i.e. for every $n$ the $n$-torsion is finite.
\end{lemm}
\begin{proof}
    Let $v \in C^{(1)}$ that does not lie above the $x$-adic valuation. Note that the residue field $\kappa(v)$ is a finite extension of $\mathbf{C}\lau{x}$ and carries a natural valuation $\operatorname{val}_v: \kappa(v) \to \mathbf{Z}$. We define $U_v^0 \subseteq \mathcal{O}_v^{\times}$ to be the subgroup formed by elements $f$ such that $\operatorname{val}_v\left(\overline{f}\right) = 0$ where $\overline{f}\in \kappa(v)$ is the reduction of $f$. For a place $w \in C^{(1)}$ above the $x$-adic valuation we define $U_w^0$ analogously by replacing the $x$-adic valuation by the $y$-adic valuation. Let $U^0 \subseteq \Acal(\mathbf{G}_m)$ be the image of $\prod_{v \in C^{(1)}} U_v^0$. A straightforward application of Hensel's lemma ensures that $U^0$ is a divisible group. \par
    Set $C_K^0 \coloneq \left(\Acal(\mathbf{G}_m)/U^0\right)_{\mathrm{tors}}$. There is a natural morphisms $\varphi:\Acal(\mathbf{G_m} \to \Pic (C)$ given by $(f_v)_{v \in C^{(1)}} \mapsto \sum_{v\in C^{(1)}} v(f_v) \cdot v$ whose kernel is isomorphic to $\prod_{v \in C^{(1)}} \mathcal{O}_v^{\times} / R^{\times}$ where $R$ is the integral closure of $\mathbf{C}[\![x,y]\!]$ in $K$. Note that $\varphi(U^0) = 0$ and $\Ker(\varphi) / U^0 \simeq \left(\prod_{v\in C^{(1)}} \mathbf{Z} \right)/\mathbf{Z}\cdot(e_v)_{v\in C^{(1)}}$ where
    \[
    e_v = 
    \begin{cases}
        \operatorname{val}_v(x) & v\nmid x \\
        \operatorname{val}_v(y) & v\mid x.
    \end{cases}
    \]
    Note that $F \coloneq \left(\left(\prod_{v\in C^{(1)}} \mathbf{Z} \right)/\mathbf{Z}\cdot (e_v)_{v\in C^{(1)}}\right)_{\mathrm{tors}}$ is a finite group. Then, we can deduce an exact~sequence
    \[
        0 \to F \to C_K^0 \to (\Pic (C))_{\mathrm{tors}}.
    \]
    This is enough to conclude that $C_K^0$ is of cofinite type because for every $n\in \mathbf{Z}$ we have $\Pic X [n] \simeq \H^1(X,\mu_n)$ and the latter group is finite \cite[Théorème~3.1]{Izquierdo2019Dualite}.
\end{proof}



\begin{lemm} \label{lemm:TorEqualTorTor:KerRhoFiniteExponent}
    Let $\tau : S \to T$ ba a morphism of $K$-tori whose kernel $R$ is a $K$-torus and $\rho: Q \to R$ a surjection of $K$-tori. Let $\pi: Q \to S$ be the composition of $\rho$ with the natural inclusion of $R$ in $S$. Then, the kernel of the induced morphism
    \[
        \bar{\tau} \colon \frac{\Acal(S)}{\Acal(S)_{\div} + \pi(\Acal(Q)))} \longrightarrow \frac{\Acal(T)}{\Acal(T)_{\div}},
    \]
    has finite exponent.
\end{lemm}

\begin{proof}
    The homomorphism $\bar{\tau}$ is well defined because $\rho(\Acal(S)_{\div}) \subseteq \Acal(T)_{\div}$ and $\tau\circ \pi = 0$. On the other hand, by Corollary~\ref{coro:CategoryOfToriUptoIsogeny}, there exist $n_1,n_2 \ge 1$ and morphisms of tori
    \[
        \phi\colon T \longrightarrow S \qquad \text{and} \qquad \psi\colon R \longrightarrow Q
    \]
    such that
    \[
        \tau \circ \phi = [n_1]_{T} \qquad \text{and} \qquad \rho \circ \psi = [n_2]_{R}.
    \]
    We claim that $\Ker(\bar{\tau})$ is $(n_1n_2)$-torsion. Indeed, let $a \in \Acal(S)$ such that $\tau(a) \in \Acal(T)_{\div}$. By Corollary~\ref{coro:DivA(T)Surjective}, there exists $a' \in \Acal(S)_{\div}$ such that $\tau(a') = \tau(a)$. Since $a - a'$ and $a$ have the same image in
    \[
        \frac{\Acal(S)}{\Acal(S)_{\div} + \pi(\Acal(Q))},
    \]
    we may replace $a$ by $a - a'$, that is, we may assume $\tau(a) = 0$ in $\Acal(T)$. \par 

    Let $s_\Abf \in S(\Abf_K)$ be an adelic point representing $a$. Then, we have $\tau(s_\Abf) = t$ for some $t \in T(K)$ and the following equality holds
    \[
        \tau(s_\Abf^{n_1}) = t^{n_1} = \tau(\phi(t)).
    \]
    In particular, $r_\Abf \coloneq s_\Abf^{n_1} \phi(t)^{-1}$ belongs to $R(\Abf_K)$. Since $\phi(t) \in S(K)$, the class $n_1sa \in \Acal(S)$ of $s_\Abf^{n_1}$ is the image of the class $b \in \Acal(R)$ of $r_\Abf$. Now, we have $r_\Abf^{n_2} = \rho(\psi(r_\Abf))$, which~implies
    \[
        n_2b \in \rho(\Acal(Q))
    \]
    and $n_1 n_2a \in \pi(\Acal(Q))$. This proves that the class of $a$ in
    \[
        \frac{\Acal(S)}{\Acal(S)_{\div} + \pi(\Acal(Q))}
    \]
    is an element of $(n_1n_2)$-torsion.
\end{proof}


\subsection{Generalities on obstruction sets}\label{subsection generalities obstruction sets}

We now introduce various subsets of $X(\mathbf{A}_K)$ that contain the diagonal image of $X(K)$.

\begin{defn}
    Let $G$ a $K$-group and $X$ a $K$-variety. For a $X$-torsor $f:Y \to X$ under $G$ we define the descent set associated to $f$ as 
    \begin{equation*}
        X(\Abf_K)^f \coloneq \bigcup_{[\sigma] \in \H^1(K,G)} f^{\sigma}(Y^{\sigma}(\Abf_K)).
    \end{equation*}
\end{defn}

\begin{defn}
    Let $\mathcal{C}$ be a class of algebraic $K$-groups. For a $K$-variety $X$ we define the $\mathcal{C}$-descent set as
    \begin{equation*}
        X(\Abf_K)^{\mathcal{C}} \coloneq \bigcap_{G \in \mathcal{C}} \bigcap_{\substack{f:Y \to X \\G\textbf{-torsor}}} X(\Abf_K)^f.
    \end{equation*}
    We also define the iterated $\mathcal{C}$-descent set as 
    \begin{equation*}
        X(\Abf_K)^{\mathcal{C},\mathcal{C}} \coloneq \bigcap_{G \in \mathcal{C}} \bigcap_{\substack{f:Y \to X \\G\text{-torsor}}} \bigcup_{[\sigma] \in \H^1(K,G)} f^{\sigma} (Y^{\sigma}(\Abf_K)^{\mathcal{C}}).
    \end{equation*}
\end{defn}

\noindent The fact that the diagonal image of $X(K)$ is contained in $X(\mathbf{A}_K)^{\mathcal{C}}$ and  $X(\mathbf{A}_K)^{\mathcal{C},\mathcal{C}}$ for any class is a simple verification -- cf. \cite[p. 22]{Sko2001TorsorsAndRationalPoints}. \par

For the rest of the paper we focus on two classes of groups:
\begin{itemize}
    \item ``$\tor$'' denotes the class of algebraic tori, i.e. algebraic $K$-groups $T$ such that $T_{\overline{K}} \simeq \mathbf{G}_m^n$ for some $n$.
    \item ``$\mult$'' denotes the class of groups of multiplicative type, i.e. algebraic $K$-groups $M$ such that $T_{\overline{K}}$ is a subgroup of $\mathbf{G}_m^n$ for some $n$.
\end{itemize}

\begin{lemm} \label{lemm:TorTorSubgroup}
    Let $T$ be a $K$-torus. Then, the subset $T(\Abf_K)^{\tor,\tor}$ is a subgroup of $T(\Abf_K)$.
\end{lemm}

\begin{proof}
    Notice that for a a torsor $f\colon Y \to T$ under a torus $S$, the set $T(\Abf_K)^f$ is identify with the kernel of the homomorphism
    \[
        T(\Abf_K) \longrightarrow \left(\prod_{v\in C^{(1)}} \H^1(K_v,S) \right) {\big / } \Img(\H^1(K,S))
    \]
    that takes $t_{\mathbf{A}}$ to the class of $f^{-1}(t_{\mathbf{A}})$. In particular, $T(\Abf_K)^{\tor}$ is subgroup $T(\Abf_K)$.\par
    For every torsor $f\colon Y \to T$ under a torus $S$, there exists an exact sequence of $K$-tori
    \[
    1 \longrightarrow S \longrightarrow R \overset{g}{\longrightarrow} T \longrightarrow 1
    \]
    such that $Y$ is a $K$-torsor under $R$, thanks to \cite[Theorem~A.1]{ILA2021LocalGlobal}. For every $[\sigma] \in \H^1(K,S)$ such that $Y^{\sigma}(\mathbf{A}_K)^{\tor} \neq \varnothing$ we have $Y^{\sigma} \simeq R$ thanks to \cite[Theorem~6.4]{ILA2021LocalGlobal}, \cite[Corollaire~8.3]{CTH2015DualCt}, and \cite[Théorème~6.2]{Izquierdo2019Dualite}. Moreover, after trivializing $Y^{\sigma}$ the map $f$ is identified with the composite of $g$ with a translation by a point of $T(K)$. Then, 
    \[
        \bigcup_{[\sigma] \in \H^1(K,S)} f^{\sigma}(Y^{\sigma}(\Abf_K)^{\tor}) = T(K)g(R(\Abf_K)^{\tor}).
    \]
    Taking the intersection for all $K$-tori $S$ and torsors $f$, we conclude the lemma.
\end{proof}

\noindent We now introduce a slight variant of an obstruction set introduced in \cite{ILA2021LocalGlobal}. Recall that a torus $Q$ is said to be {\em quasitrivial} if there exists an étale algebra $E/K$ such that $Q \simeq \Res_{E/K} \mathbf{G}_{m,E}$ where $\Res_{E/K}$ denotes the Weil restriction with respect to $E/K$.
\begin{defn}
    Let $X$ be a $K$-variety and $B \in \{\Be, \Br_1, \Br\}$. Define
    \[
        X(\Abf_K)^{\qt, B} \coloneq \bigcap_{\substack{f \colon Y \xrightarrow{Q}X \\ Q\text{ quasitrivial} }} f(Y(\Abf_K)^{B}).
    \]
\end{defn}

\begin{remk}
    In the case $B = \Be$, the notation coincides with the one in \cite[Definition~6.2]{ILA2021LocalGlobal}. However, our definition corresponds to the ``more natural'' related obstruction mentioned in \cite[Remark~6.3]{ILA2021LocalGlobal}, which is {\em a priori} stronger than the one defined in \cite[Definition~6.2]{ILA2021LocalGlobal}.
\end{remk}

\noindent A straightforward verification shows that all of these obstructions are functorial. For instance, for every morphism $f \colon X \to Y$ we have 
\begin{align*}
    f(X(\mathbf{A}_K)^{\qt,B}) & \subseteq Y(\mathbf{A}_K)^{\qt,B} \\
    f(X(\mathbf{A}_K)^{\tor}) & \subseteq Y(\mathbf{A}_K)^{\tor}\\
    f(X(\mathbf{A}_K)^{\tor,\tor}) & \subseteq Y(\mathbf{A}_K)^{\tor,\tor}.
\end{align*}
\section{Torsors of prescribed extended types}\label{sec:TorsorOfGivenXType}

\subsection{Extended types and quasitrivial realizations} \label{subsec:QuasitrivialRealization}

Let $X$ be a smooth geometrically integral variety defined over a field $k$. As in \cite{BvH2009UPic}, we consider the following two-term complexes of Galois modules
\[
    \KD(\bar{X}) \coloneq \left[\bar{k}(X)^\times \longrightarrow \Div(\bar{X})\right] \qquad \text{and} \qquad \UPic(\bar{X}) \coloneq  \left[\bar{k}(X)^\times/\bar{k}^\times \longrightarrow \Div(\bar{X})\right],
\]
both concentrated in degrees $-1$ and $0$. The cohomology of $\UPic(\bar{X})$ is given by
\[
    \H^{-1}(\UPic(\bar{X})) = \bar{k}[X]^\times/\bar{k}^\times \qquad \text{and} \qquad \H^0(\UPic(\bar{X})) = \Pic(\bar{X}).
\]
The natural morphism $\KD(\bar{X}) \to \UPic(\bar{X})$ fits in a distinguished triangle in the bounded derived category $\Dcal^+(k)$:
\begin{equation} \label{eq:TriangleKDAndUPic}
    \bar{k}^\times[1] \longrightarrow  \KD(\bar{X}) \longrightarrow \UPic(\bar{X}) \longrightarrow \bar{k}^\times[2].
\end{equation}
\noindent Let $M$ be any smooth $k$-group of multiplicative type, i.e. the order of $\widehat{M}_{\tors}$ is invertible in $k$. The following result is established in \cite[Proposition~8.1 and Definition~8.2]{HS2013Descent}.

\begin{prop} \label{prop:ExtendedType}
    Applying $\RHom_{\Dcal^+(k)}\bigl(\widehat{M},-\bigr)$ to \eqref{eq:TriangleKDAndUPic} yields a functorial exact sequence
    \[
        \H^1(k,M) \longrightarrow \H^1(X,M) \xrightarrow{\XType} \Hom_{\Dcal^+(k)}\bigl(\widehat{M},\UPic(\bar{X})\bigr) \overset{\ob}{\longrightarrow} \H^2(k,M) \longrightarrow \H^2(X,M).
    \]
    More precisely, there is a natural isomorphism 
    \[
        \H^1(X,M) \cong \Hom_{\Dcal^+(k)}\bigl(\widehat{M},\KD(\bar{X})\bigr),
    \]
    and the map $\XType$ is induced by the morphism $\KD(\bar{X}) \to \UPic(\bar{X})$.
\end{prop}

\noindent We define the \textit{extended type} of an $X$-torsor $Y\to X$ under $M$ as $\XType([Y\to X])$. Thus, the extended type of an $X$-torsor under $M$ classifies it up to Galois twist. On the other hand, the precise obstruction to the existence of torsors of with given extended type 
\[
    \lambda\colon \widehat{M} \longrightarrow \UPic(\bar{X})
\]
is the class $\ob(\lambda) \in \H^2(k,M)$. In Proposition~\ref{prop:QuasitrivialRealization}, we give an arithmetic interpretation of this realization problem as the existence of $k$-rational points on a torsor under a torus. \par
Let us compute $\UPic(\bar{T})$ for a $K$-torus $T$ since it is relevant for the a further argument:

\begin{exam} \label{exam:TorsorUnderTorus}
    If $X$ is a $k$-torsor under a torus $T$, then $\Pic(\bar{X}) = 0$ and hence 
    \[
        \UPic(\bar{X}) \simeq \bar{k}[X]^\times/\bar{k}^\times[1].
    \]
    A twisted version of Rosenlicht's lemma \cite[Proposition~3]{Rosenlicht1957Rationality} gives an exact sequence
    \begin{equation} \label{eq:RoselichtLemma}
        1 \longrightarrow \bar{k}^\times \longrightarrow \bar{k}[X]^\times \overset{f \mapsto \chi_f}{\longrightarrow} \widehat{T} \longrightarrow 0,
    \end{equation}
    Here, for each invertible function $f \in \bar{k}[X]^\times$, the associated character 
    \[
        \chi_f \colon T \to \Gbf_m
    \]
    is the unique one such that
    \[
        f(t \cdot x) = \chi_f(t) f(x)
    \]
    for all $t \in T(\bar{k})$ and $x \in X(\bar{k})$. In particular, one has
    \[
        \UPic(\bar{X}) \simeq \bar{k}[X]^\times/\bar{k}^\times[1] = \widehat{T}[1].
    \]
    On the other hand, it is worth nothing that under the isomorphism
    \[
        \Hom_{\Dcal^+(k)}\bigl(\widehat{T}, \UPic(\Spec(\bar{k}))\bigr) = \Ext^1_k\bigl(\widehat{T}, \bar{k}^\times\bigr) \cong \H^1(k,T),
    \]
    the element $\XType([X])$ corresponds to the class of extension \eqref{eq:RoselichtLemma} and $-[X] \in \H^1(k,T)$, see \cite[Proposition~1.4.2 and Proposition~1.4.3]{CTS1987Descente} for a proof.
\end{exam}

\noindent The following observation appeared in \cite[Section~3]{HW2020Galois}: suppose that we are given a torsor $Z \to X$ under $M$ of extended type $\lambda\coloneq\XType(Z)$. Take any exact sequence
\[
    1 \longrightarrow M \overset{\iota}{\longrightarrow} Q \overset{\pi}{\longrightarrow} T \longrightarrow 1,
\]
where $Q$ and $T$ are $k$-tori, and $Q$ is quasitrivial. Let
\[
    Y \coloneq Z \times^M Q
\]
which is an $X$-torsor under $Q$ of extended type $\XType(Y) = \lambda \circ \widehat{\iota}$. Since $Z \to X$ is a torsor under $M$, the $T$ torsor $Y / Q \to X$ is trivial. Choosing any isomorphism $Y/ Q \simeq X \times T$ induces a $\pi$-equivariant morphism 
\[
    f \colon Y \to Q/M = T.
\]
For any $t \in T(k)$, the fiber $\pi^{-1}(t)$ is a $k$-torsor under $M$ that represents the image of $t$ by the connecting homomorphism
\[
    T(k) \longrightarrow \H^1(k,M).
\]
It follows that $f^{-1}(t) = Y \times^M \pi^{-1}(t)$ is an $X$-torsor under $M$. More precisely, it is a Galois twist of $Y$ and hence has extended type $\lambda$ by Proposition~\ref{prop:ExtendedType}.
\medskip

\noindent
Now, if we are given an extended type $\lambda\colon \widehat{M} \to \UPic(\bar{X})$ without any torsor realizing it, we can still pursue the above idea as follows: embed $M$ into a quasitrivial torus $Q$ such that the pushforward extended type $\widehat{Q} \to \UPic(\bar{X})$ can be realized by a torsor 
\[
    Y \to X.
\]
Let $T\coloneq Q/M$ and let $\pi\colon Q \to T$ be the projection. Then, we can construct a $\pi$-equivariant morphism $f$ from $Y$ to a {\em $k$-torsor under $T$}. The existence of $k$-rational point of this torsor will then be equivalent to the existence of $X$-torsors under $M$ of extended type $\lambda$.
\medskip

\begin{prop} \label{prop:QuasitrivialRealization}
    For any morphism $\lambda\colon \widehat{M} \to \UPic(\bar{X})$ in the derived category of Galois modules, the following data exist:

    \begin{enumerate}
        \item \label{prop:QuasitrivialRealization:1} An embedding $\iota \colon M \hookrightarrow Q$ where $Q$ is a quasitrivial $k$-torus such that the restriction $\iota_*(\ob(\lambda)) \in \H^2(k,Q)$ vanishes,

        \item \label{prop:QuasitrivialRealization:2} a $k$-torsor $E$ under the torus $T\coloneq Q/M$ such that the connecting homomorphism
        \[
            \delta\colon \H^1(k,T) \longrightarrow \H^2(k,M)
        \]
        takes $[E]$ to $\ob(\lambda)$,
        
        \item \label{prop:QuasitrivialRealization:3} a torsor $p \colon Y \to X$ under $Q$ of extended type $\XType(p) = \lambda \circ \widehat{\iota}$, and

        \item \label{prop:QuasitrivialRealization:4} a $\pi$-equivariant morphism $f \colon Y \to E$, where $\pi\colon Q \to T$ is the projection. 
    \end{enumerate}
\end{prop}
\begin{proof}
    By \cite[Lemma~3.7]{BCTS08ElementaryObstruction}, there is an embedding $\iota\colon M \hookrightarrow Q$ satisfying~\ref{prop:QuasitrivialRealization:1}.\par

    Since $\iota_*(\alpha) = 0$, the long exact sequence associated with
    \[
        1 \longrightarrow M \overset{\iota}{\longrightarrow} Q \overset{\pi}{\longrightarrow} T \longrightarrow 1
    \]
    ensures the existence of a $k$-torsor $E$ under $T$ such that $\delta([E]) = \ob(\lambda)$ as in~\ref{prop:QuasitrivialRealization:2}. \par
    
    Under the isomorphism $\H^1(k,T) \cong \Ext^1_k\bigl(\widehat{T},\bar{k}^\times\bigr)$, the class of the torsor $E$ corresponds to a homomorphism
    \[
        \varepsilon\colon \widehat{T} \longrightarrow \bar{k}^\times[1]
    \]
    in $\Dcal^+(k)$ that fits in the right square of the diagram of distinguished triangles
    \begin{equation} \label{eq:QuasitrivialRealization}
    \xymatrix{
        \widehat{T} \ar[r]^{\widehat{\pi}} \ar[d]^{\varepsilon} & \widehat{Q} \ar[r]^{\widehat{\iota}} \ar@{-->}[d]^{\eta} & \widehat{M} \ar[r]^{d} \ar[d]^{\lambda} & \widehat{T}[1] \ar[d]^{\varepsilon[1]}\\
        \bar{k}^\times[1] \ar[r] & \KD(\bar{X}) \ar[r] & \UPic(\bar{X}) \ar[r] & \bar{k}^\times[2].
    }
    \end{equation}
    Indeed, the connecting homomorphism $\delta$ is induced by the morphism $T \to M[1]$ that is dual to $d$. By the axioms of triangulated categories, there exists a dashed morphism
    \[
        \eta\colon \widehat{Q} \longrightarrow \KD(\bar{X})
    \]
    that makes \eqref{eq:QuasitrivialRealization} a morphism of distinguished triangles. It follows from Proposition~\ref{prop:ExtendedType} that $\eta$ corresponds to a torsor $p \colon Y \to X$ under $Q$. Commutativity of the middle square ensures condition~\ref{prop:QuasitrivialRealization:3}, namely $\XType(p) = \lambda \circ \hat{\iota}$.
    \medskip

    Lastly, commutativity of the left square in \eqref{eq:QuasitrivialRealization} implies the classes
    \[
        [Y] \in \H^1(X,Q) \qquad \text{and} \qquad [E] \in \H^1(k,T)
    \]
    have the same image in $\H^1(X,T)$ or, equivalently, there exists a $T$-equivariant isomorphism
    \[
        Y/M = Y \times^Q T \overset{\cong}{\longrightarrow} X \times E
    \]
    over $X$. Composing it with the projection $Y \to Y/M$ yields a torsor
    \[
        Y \longrightarrow X \times E    
    \]
    under $M$. Projecting onto $E$ gives a morphism $f \colon Y \to E$ that fulfills condition~\ref{prop:QuasitrivialRealization:4}.
\end{proof}

\begin{defn} \label{defn:QuasitrivialRealization}
    Let $\lambda \colon \widehat{M} \to \UPic(\bar{X})$ be a morphism in $\Dcal^+(k)$. A {\em quasitrivial realization} of $\lambda$ is any quintuple $(Q, E, Y, f,p)$ satisfying the conditions of Proposition~\ref{prop:QuasitrivialRealization}.
\end{defn}

\begin{prop} \label{prop:TorsorOfGivenXType}
    Let $X$ be a smooth geometrically integral $k$-variety, $M$ a smooth $k$-group of multiplicative type, and $\lambda \in \Hom_{\Dcal^+(k)}\bigl(\widehat{M},\UPic(\bar{X})\bigr)$. Let $(Q, E, Y, f,p)$ be a quasitrivial realization of $\lambda$. Then, there exists a $X$-torsor of extended type $\lambda$ if and only if $E(k) \neq \varnothing$.
\end{prop}
\begin{proof}
    Condition~\ref{prop:QuasitrivialRealization:2} of  Proposition~\ref{prop:QuasitrivialRealization} ensures $\delta([E]) = \ob(\lambda)$. On the other hand, the connecting homomorphism $\delta$ is injective because $\H^1(k,Q) = 0$. Then, $\ob(\lambda)$ vanishes if and only if $[E] = 0$. In other words, torsors of extended type $\lambda$ exist if and only if $E(k) \neq \varnothing$.
\end{proof}

\begin{remk} 
    The chosen isomorphism $Y/M \cong X \times E$ of $X$-torsor under $T$ is by no mean canonical. By careful computations with derived categories, we can actually show the following. There exists a morphism $\psi\colon X \to T$ (that is, an $X$-point of $T$) such that, after modifying the isomorphism $Y/M \cong X \times E$ by the translation automorphism of $X \times E$ induced by $\psi$, the morphism $f$ yields a family of torsors of extended type $\lambda$ parametrized by $E$. In particular, for any $k$-rational point $e \in E(k)$, the fibre $f^{-1}(e)$ is an $X$-torsor under $M$ with $\XType(f^{-1}(e)) = \lambda$.
\end{remk}


\subsection{Realization over geometric global fields}

For the rest of the section $K$ denotes a geometric global field as defined in Section~\ref{subsection generalities fields}.
        

\begin{defn}
    Let $X$ be a geometrically integral smooth $K$-variety, $\lambda: \widehat{M} \to \UPic(\bar{X})$ a morphism in $\Dcal^+(K)$, and $(Q,E,Y,f,p)$ a quasitrivial realization of $\lambda$. We define the group $\Be_{\lambda}(Y) \subseteq \Br_1(Y)$ as 
    \[
    \Be_{\lambda}(Y) \coloneq f^*(\Be(E)).
    \]
    We also define
    \[
        X(\Abf_K)^{p,\Be_{\lambda}} \coloneq p(Y(\Abf_K)^{\Be_{\lambda}}).
    \]
\end{defn}

\begin{prop}\label{prop qt Be_lamba existence of type lambda}
    Let $K$ be a geometric global field, $X$ a smooth geometrically integral $K$-variety, $\lambda \in \Hom_{\Dcal^+(K)}(\widehat{M}, \UPic(\bar{X}))$, and $(Q,E,Y,f,p)$ a quasitrivial realization. Assume that $X(\Abf_K) \neq \varnothing$. Then, $X$-torsors of extended type $\lambda$ exist if and only if we have $X(\Abf_K)^{p,\Be_\lambda} \neq \varnothing$.
\end{prop}

\begin{proof}
Proposition~\ref{prop:TorsorOfGivenXType} ensures that there exists a torsors of type $\lambda$ if and only if  we have $E(K) \neq \varnothing$. On the other hand, the Brauer-Manin obstruction with respect to $\Be$ is the only obstruction to the local-global principle for $E$, thanks to \cite[Corollaire~8.3]{CTH2015DualCt} and \cite[Théorème~5.2]{Izquierdo2019Dualite}. Then, the condition $E(K) \neq \varnothing$ is equivalent to $E(\Abf_K)^{\Be(E)} \neq \varnothing$. \par
Consider $x_{\Abf} \in X(\Abf_K)$ and $\alpha \in \Be(E)$. Note that since $Q$ is a quasitrivial torus, there exists $y_{\Abf} \in Y(\Abf_K)$ lifting $x_{\Abf}$. By functoriality of the Brauer-Manin pairing, we have
\begin{equation}\label{eq compatibility  BM}
    \operatorname{BM}(y_{\Abf}, f^*\alpha) = \operatorname{BM}(f(y_{\Abf}) , \alpha).
\end{equation}
Note that for every $e_{\Abf}, e'_{\Abf} \in E(\Abf_K)$, we have $\operatorname{BM}(e_{\Abf} , \alpha) = \operatorname{BM}(e_{\Abf}' , \alpha)$ because $\alpha$ belongs to $\Be(E)$. Then, there exists a point $e_{\Abf} \in E(\Abf_K)$ orthogonal to $\Be(E)$ if and only if we have $Y(\Abf_K)^{\Be_{\lambda}(Y)} \neq \varnothing$.
\end{proof}

\noindent We obtain the following analogue of \cite[Corollary~8.17]{HS2013Descent} in the setting of geometric global~fields.

\begin{coro} \label{coro:QtBeNonempty}
    Let $K$ be a geometric global field and $X$ a smooth geometrically integral variety over $K$. If $X(\Abf_K)^{\qt,\Be} \neq \varnothing$, then there exists a $X$-torsor of any given extended type, that is, the~homomorphism
    \[
        \XType\colon \H^1(X,M) \longrightarrow \Hom_{\Dcal^+(K)}\bigl(\widehat{M},\UPic(\bar{X})\bigr)
    \]
    is surjective for any $K$-group of multiplicative type $M$.
\end{coro}

\begin{proof}
    Let $\lambda$ be an element of $\Hom_{\Dcal^+(K)}\bigl(\widehat{M},\UPic(\bar{X})\bigr)$. The proposition follows directly from the existence of a quasitrivial realization and Proposition~\ref{prop qt Be_lamba existence of type lambda}.
\end{proof}


\section{Idempotency of toric descent} \label{sec:TorEqualTorTor}

The main objective of this section is to prove the following result:
\begin{thm} \label{thm:TorEqualTorTor}
    Let $X$ be a smooth geometrically integral variety defined over a geometric global field $K$ such that $\Pic(\bar{X})$ is finitely generated. Then
    \[
        X(\Abf_K)^{\tor,\tor} = X(\Abf_K)^{\qt,\Br_1} = X(\Abf_K)^{\mult} = X(\Abf_K)^{\tor}.
    \]
\end{thm}


\subsection{Tori versus groups of multiplicative type}

In this subsection, we establish the following relation between various descent sets of adelic points. In particular, we show that torsors under groups of multiplicative type do not refine the toric descent set.

\begin{thm} \label{thm:TorEqualMult}
    Let $X$ be a smooth geometrically integral variety defined over a two-dimensional geometric global field $K$. Then, we have
    \[
        X(\Abf_K)^{\tor,\tor} \subseteq X(\Abf_K)^{\qt,\Br_1} \subseteq X(\Abf_K)^{\mult} = X(\Abf_K)^{\tor}.
    \]
\end{thm}

\noindent The proof of Theorem~\ref{thm:TorEqualMult} relies on Proposition~\ref{prop:TorEqualMult} below. We first recall the following general fact: let $f \colon Y \to X$ be a torsor under a $K$-torus $T$. Set
    \[
        \Br_1(Y/X) \coloneq \Ker\bigl(\Br(Y) \longrightarrow \Br(\bar{Y})/f^\ast \Br(\bar{X}) \bigr),
    \]
which contains $\Br_1(Y)$. By \cite[(2.4) and Proposition~2.3]{HW2020Galois}, there exists a morphism
    \begin{equation} \label{eq:TheMapPhi}
        \varphi \colon  \Br_1(Y/X) \longrightarrow \H^2(K, \widehat{T})
    \end{equation}
with the following property: for any field extension $L/K$, algebraic Brauer class $\alpha \in \Br_1(Y)$, and $L$-rational points $y \in Y(L)$ and $t \in T(L)$, we have
    \begin{equation} \label{eq:EvaluationOfAlgBrauerClassOnTorsor}
        \alpha(t \cdot y) = \alpha(y) + t \cup \varphi(\alpha).
    \end{equation}

\begin{prop} \label{prop:TorEqualMult}
    For an adelic point $x_\Abf \in X(\Abf_K)$, the following are equivalent.
    
    \begin{enumerate}
        \item \label{prop:TorEqualMult:1} $x_\Abf \in X(\Abf_K)^{\mult}$.

        \item \label{prop:TorEqualMult:2} $x_\Abf \in X(\Abf_K)^{\tor}$.

        \item \label{prop:TorEqualMult:3} 
        Let $Y \to X$ be a torsor under a quasitrivial $K$-torus $Q$ with the associated~morphism
        \[
            \varphi\colon \Br_1(Y) \longrightarrow \H^2(K,\widehat{Q})
        \]
        
        given by \eqref{eq:TheMapPhi}. Then, any adelic point $y_\Abf \in Y(\Abf_K)$ lifting $x_\Abf$ is orthogonal to the group $\varphi^{-1}\bigl(\Sha^2(K,\widehat{Q})\bigr)$.

        \item \label{prop:TorEqualMult:4}
        Let $Y \to X$ and $\varphi$ be as in \ref{prop:TorEqualMult:3} and $y_\Abf\in Y(\Abf_K)$ be any adelic point lifting $x_\Abf$. Then, there exists a collection of adelic points
            \[
                q_{\Abf}^{(n)} \in Q(\Abf_K)
            \]
        indexed by $n \ge 1$ that satisfies the following conditions.
        
        \begin{enumerate}
            \item \label{prop:TorEqualMult:4a} $q_{\Abf}^{(n)} \cdot y_\Abf$ is orthogonal to $\varphi^{-1}\bigl(\tensor[_n]{\H}{^2}(K,\widehat{Q}) \bigr)$ for all $n \ge 1$.
            
            \item \label{prop:TorEqualMult:4b} $q_\Abf^{(n)} \equiv q_\Abf^{(m)} \pmod{n Q(\Abf_K)}$ whenever $n \mid m$.
        \end{enumerate}
    \end{enumerate}
\end{prop}

\begin{proof}
    The implication~\ref{prop:TorEqualMult:1}$\Rightarrow$\ref{prop:TorEqualMult:2} is obvious.\par   
    
    Assume~\ref{prop:TorEqualMult:2}. Let $\alpha \in \Br_1(Y)$ belong to $\varphi^{-1}(\Sha^2(K,\widehat{Q}))$. By \cite[Theorem~6.4]{ILA2021LocalGlobal}, there exists a lift $y_\Abf^{\alpha} \in Y(\Abf_K)$ of $x_\Abf$ that is orthogonal to $\alpha$. We have $y_\Abf^{\alpha} = q_\Abf^{\alpha} \cdot y_\Abf$ for some $q_\Abf^{\alpha} \in Q(\Abf_K)$. Since $\varphi(\alpha)|_{K_v} = 0$ in $\H^2(K_v,\widehat{Q})$ for all $v$, it follows from \eqref{eq:EvaluationOfAlgBrauerClassOnTorsor} that
    \[
        \sum_v \alpha(y_v) = \sum_v \alpha(y_v^\alpha) - q_v^\alpha \cup \bigl(\varphi(\alpha)|_{K_v}\bigr) = \sum_v \alpha(y_v^\alpha) = 0. 
    \]
    This proves~\ref{prop:TorEqualMult:3}.
    \par
    
    Assume~\ref{prop:TorEqualMult:3}. Then, the functional homomorphism
    \[
        \varepsilon\colon \Br_1(Y) \longrightarrow \Qbf/\Zbf, \qquad \alpha \longmapsto \sum_v \alpha(y_v)
    \] 
    vanishes on $\varphi^{-1}\bigl(\Sha^2(K,\widehat{Q})\bigr)$. Since $\varphi$ induces an injective homomorphism
    \[
        \frac{\Br_1(Y)}{\varphi^{-1}\bigl(\Sha^2(K,\widehat{Q})\bigr)} \longhookrightarrow \frac{\H^2(K,\widehat{Q})}{\Sha^2(K,\widehat{Q})}
    \]
    and $\Qbf/\Zbf$ is a divisible abelian group, then there exists a homomorphism
    \[
        \varepsilon'\colon \H^2(K,\widehat{Q}) \longrightarrow \Qbf/\Zbf
    \]
    such that $\varepsilon'|_{\Sha^2(K,\widehat{Q})} = 0$ and $\varepsilon' \circ \varphi = \varepsilon$. Proposition~\ref{prop properties adelic spaces}~\ref{prop properties adelic spaces: Poitou-Tate} provides adelic points
    \[
        q_{\Abf}^{(n)}  \in Q(\Abf_K),
    \]
    for $n \ge 1$, that satisfy condition~\ref{prop:TorEqualMult:4b} and
    \[
        \varepsilon'(\beta) = -\sum_{v} q_v^{(n)} \cup \bigl(\beta|_{K_v}\bigr)
    \]
    for all $n \ge 1$ and $\beta \in \tensor[_n]{\H}{^2}(K,\widehat{Q})$. For every $\alpha \in \varphi^{-1}\bigl(\tensor[_n]{\H}{^2}(K,\widehat{Q})\bigr)$ we have
    \[
        \sum_v \alpha(q_v^{(n)} \cdot y_v) = \sum_v \alpha(y_v) + q_v^{(n)} \cup \bigl(\varphi(\alpha)|_{K_v}\bigr) = \varepsilon(\alpha) - \varepsilon'(\varphi(\alpha)) = 0
    \]
  by construction of $\varepsilon'$, which establishes condition~\ref{prop:TorEqualMult:4a}.
    \par
    
    Now assume that condition~\ref{prop:TorEqualMult:4} holds. Let $f\colon Z \to X$ be a torsor under a $K$-group of multiplicative type $M$. Take any exact sequence
    \[
        1 \longrightarrow M \longrightarrow Q \overset{\pi}{\longrightarrow} T \longrightarrow 1,
    \]
    where $Q$ and $T$ are $K$-tori and $Q$ is quasitrivial. Let $g\colon Y \to X$ be the torsor under $Q$ defined by $Y \coloneq Z \times^M Q$. Note that the adelic point $x_\Abf$ admits a lift $y_\Abf \in Y(\Abf_K)$ since $\H^1(K_v,Q) = 0$ for all $v \in C^{(1)}$. Additionally, the torsor $Y/M \to X$ under $T$ is trivial, so there exists a trivialisation $Y/M \simeq X \times T$. Let
    \[
        h \colon Y \to T
    \]
    be the composition of the quotient $Y \to Y/M$ and the projection on the se= (cond factor. Note that $h$ is $\pi$-equivariant. Set
    \[
        t_v^{(n)} \coloneq h(q_v^{(n)} \cdot y_v) = \pi(q_v^{(n)}) h(y_v) \qquad \text{and} \qquad t_\Abf^{(n)} = \bigl(t_v^{(n)}\bigr)_v \in T(\Abf_K),
    \]
    where $q_v^{(n)}$ comes from \ref{prop:TorEqualMult:4}. By condition~\ref{prop:TorEqualMult:4a} and functoriality of the homomorphism \eqref{eq:TheMapPhi}, the adelic point $t_\Abf^{(n)}$ is orthogonal to $\tensor[_n]{\H}{^2}(K,\widehat{T})$. Condition~\ref{prop:TorEqualMult:4b} ensures that the collection  $\bigl(t_\Abf^{(n)}\bigr)_n$ defines an element of $T(\Abf_K)_\wedge$. We now consider the Poitou--Tate sequence from Proposition~\ref{prop properties adelic spaces}~\ref{prop properties adelic spaces: Poitou-Tate}. The image of $\bigl(t_\Abf^{(n)}\bigr)_n$ in $\Hom\bigl(\H^2(K,\widehat{T}),\Qbf/\Zbf\bigr)$ is zero because $\H^2(K,\widehat{T}) = \varinjlim_n \tensor[_n]{\H}{^2}(K,\widehat{T})$. Then, Proposition~\ref{prop properties adelic spaces}~\ref{prop properties adelic spaces: Poitou-Tate} gives $K$-rational points $t^{(n)} \in T(K)$, for $n \ge 1$, such that 
    \[
        t^{(n)} \equiv t_\Abf^{(n)} \pmod{nT(\Abf_K)}.
    \]
    Let $N \ge 1$ be an integer such that $N \cdot \H^1(K,M) = 0$. Such an integer exists because over a some finite extension $L/K$, one has
    \[
        M_L \cong F \times \Gbf_{\m,L}^r,  
    \]
    where $F$ is a finite commutative $L$-group scheme and $r \ge 0$. Since $\H^1(L,\Gbf_{\m}) = 0$ by Hilbert's Theorem 90, a restriction--corestriction argument shows that one can take $N = |F| \cdot [L:K]$. Since $t^{(N)} \equiv t_v^{(N)} \pmod{N T(K_v)}$ for all $v$, we have
    \[
        \delta (t^{(N)})|_{K_v}= \delta_v(t_v^{(N)}),
    \]
    where $\delta: T(K) \to \H^1(K,M)$ and $\delta_v : T(K_v) \to \H^1(K_v,M)$ are the connecting homomorphisms. Let $\sigma \in \operatorname{Z}^1(K,M)$ be a cocycle such that $[\sigma] = -\delta(t^{(N)})$. 
    \par
    
    We claim that $x_v$ lifts to $Z^{\sigma}(K_v)$ for all $v$. Indeed, the fibre $\pi^{-1}(t_v^{(N)})$ is a $K_v$-torsor under $M$ with $[\pi^{-1}(t_v^{(N)})] = \delta(t_v^{(N)})$ in $\H^1(K_v,M)$. Thus
    \[
        Y_{t_v^{(N)}}=f^{-1}(t_v^{(N)}) = Z \times^Q \pi^{-1}(t_v^{(N)})
    \]
    is an $X_{K_v}$-torsor under $M$ with
    \[
        [Y_{t_v^{(N)}}] = [Z_{K_v}] - p^\ast [\sigma]_{K_v} = [Z^\sigma_{K_v}]    
    \]
    in $\H^1(X_{K_v},M)$, where $p \colon X \to \Spec(K)$ denotes the structure morphism. But $x_v$ lifts to a $K_v$-point on $Y_{t_v^{(N)}}$ (namely $q_v^{(N)} \cdot y_v$), so $x_v$ lifts to a $K_v$-point on $Z^\sigma$ as well. This is enough to conclude \ref{prop:TorEqualMult:1} given Lemma~\ref{lemma local lift implies adelic lift} below.
\end{proof}

\begin{lemm}\label{lemma local lift implies adelic lift}
    Let $X$ be a smooth geometrically integral $K$ variaty and $f:Y \to X$ a torsor under a linear $K$ group $G$. Then, we have the following equality of subsets of $\prod_{v \in C^{(1)}} X(K_v)$:
    \[
    f(Y(\Abf_K)) = \left(\prod_{v \in C^{(1)}} f(Y(K_v)) \right) \cap X(\Abf_K).
    \]
\end{lemm}
\begin{proof}
    The contention $f(Y(\Abf_K)) \subseteq \left(\prod_{v \in C^{(1)}} f(Y(K_v)) \right) \cap X(\Abf_K)$ is clear. Let $x_{\Abf} \in X(\mathbf{A}_K) \cap \prod_{v \in C^{(1)}} f(Y(K_v))$. Let $G^{\circ}$ be the connected component of $G$ and denote by $G^{\mathrm{f}} =G/G^{\circ}$ the group of connected components. We can find an open subset  $U \subseteq C$ such that the following holds: 
        \begin{itemize}
            \item $G$ extends to a group scheme $\mathcal{G}$ over $U$, 
            \item $X$ extends to a smooth scheme $\mathcal{X}$ over $U$,
            \item $x_v \in \mathcal{X}(\mathcal{O}_v)$ for every $v \in U$,
            \item $Y$ extends to a $\mathcal{X}$-torsors $\mathcal{Y}$ under $\mathcal{G}$ over $U$, and
            \item If we denote by $\mathcal{G}^{\circ}$ and $\mathcal{G}^{\mathrm{f}}$ group scheme extending $G^{\circ}$ and $G^{\mathrm{f}}$ respectively there is an exact sequence
            \[
            1 \to \mathcal{G}^{\circ} \to \mathcal{G} \to \mathcal{G}^{\mathrm{f}} \to 1.
            \]
        \end{itemize}
        We now prove that for every $v \in U$ the $\mathcal{O}_v$-torsor $f^{-1}(x_v)$ is trivial. Note that the quotient of $f^{-1}(x_v)$ by the action of $\mathcal{G}^0 \otimes_U \mathcal{O}_v$ is a torsors under $\mathcal{G}^{\mathrm{f}}$ that has a point over $K_v$. We can apply the valuative criterion of properness to deduce that the quotient admits an $\mathcal{O}_v$-point. Therefore, the class of $f^{-1}(x_v)$ in $\H^1(\mathcal{O}_v, \mathcal{G})$ comes from $\H^1(\mathcal{O}_v, \mathcal{G}^{\circ})$. However, every $\mathcal{O}_v$-torsor under a smooth connected group is trivial thanks to Steinberg's theorem \cite[Ch III \S 2.3 Théorème 1']{Serre1994CohomologieGaloisienne} and Hensel's lemma. We conclude that there exists a lift $y_{\Abf} \in \prod_{v \in U} \mathcal{Y}(\mathcal{O}_v) \times \prod_{v \in C^{(1)} \setminus U }Y(K_v)$ of $x_{\Abf}$ along $f$.
\end{proof}

\begin{proof}[Proof of Theorem~\ref{thm:TorEqualMult}]
    It follows immediately from Proposition~\ref{prop:TorEqualMult} that
    \[
        X(\Abf_K)^{\qt,\Br_1} \subseteq X(\Abf_K)^{\mult} = X(\Abf_K)^{\tor}.
    \]
    Let $Y \to X$ be a torsor under a quasitrivial $K$-torus $Q$. By definition, any $x_\Abf \in X(\Abf_K)^{\tor,\tor}$ lifts to an adelic point $y_\Abf \in Y(\Abf_K)^{\tor}$ (note that there is no nontrivial Galois twist of $Y$). On the other hand, it is a consequence of \cite[Theorem 6.4]{ILA2021LocalGlobal} that $y_\Abf \in Y(\Abf_K)^{\Br_1}$. This shows that $X(\Abf_K)^{\tor,\tor} \subseteq X(\Abf_K)^{\qt,\Br_1}$.
\end{proof}

\subsection{Varieties with \texorpdfstring{$\Pic(\bar{X}) = 0$}{trivial Picard group}}

This subsection is devoted to the proof of Theorem~\ref{thm:TorEqualTorTor} for varieties with vanishing geometric Picard group, which is a key step for the proof of Theorem~\ref{thm:TorEqualTorTor}

\begin{thm} \label{thm:TorEqualTorTorTrivialPicardGroup}
    Let $X$ be a smooth geometrically integral variety over $K$ with $\Pic(\bar{X}) = 0$. Then, we have
    \[
        X(\Abf_K)^{\tor,\tor} = X(\Abf_K)^{\qt,\Br_1} = X(\Abf_K)^{\mult} = X(\Abf_K)^{\tor}.
    \]
\end{thm}

\noindent We start with the case of tori. In doing so, we also give a very concrete description of the toric descent set of a torus.

\begin{prop} \label{prop:TorEqualTorTorForTori}
    For any $K$-torus $T$, the following are true:
    \begin{enumerate}
        \item \label{prop:TorEqualTorTorForTori:1} $\Acal(T)_{\div} = T(\Abf_K)^{\tor}/T(K)$, and

        \item \label{prop:TorEqualTorTorForTori:2} $T(\Abf_K)^{\tor} = T(\Abf_K)^{\tor,\tor}$.
    \end{enumerate}
\end{prop}
\begin{proof}
    We first prove~\ref{prop:TorEqualTorTorForTori:1}. Let $t_\Abf \in T(\Abf_K)^{\tor}$. By \cite[Theorem~6.4]{ILA2021LocalGlobal}, $t_\Abf$ is orthogonal to $\Br_1(T)=\H^2(K,\widehat{T})$. It follows from Proposition~\ref{prop properties adelic spaces}~\ref{prop properties adelic spaces: Poitou-Tate} that for all $n \ge 1$, there exists $t^{(n)} \in T(K)$ with
    \[
        t^{(n)} \equiv t_\Abf \pmod{nT(\Abf_K)}.
    \]
    In particular, $[t_\Abf]$ belongs to $n \Acal(T)$. By virtue of Proposition~\ref{prop properties adelic spaces}~\ref{prop properties adelic spaces: div = infity div}, the class $[t_\Abf]$ belongs to $\Acal(T)_{\div}$. Hence,
    \[
        T(\Abf_K)^{\tor}/T(K) \subseteq \Acal(T)_{\div}.
    \]
    
    Conversely, let $t_\Abf = (t_v)_v \in T(\Abf_K)$ and assume its class $[t_\Abf] \in \Acal(T)_{\div}$. Let
    \[
        p\colon Y \longrightarrow T
    \]
    be a torsor under a $K$-torus $S$. By \cite[Theorem~A.1]{ILA2021LocalGlobal} (see also \cite[Theorem~1.2]{FILA2023Composition}), there exists an exact sequence
    \begin{equation}\label{eq extension tori}
     1 \to S \to R \xrightarrow{\pi} T \to 1
    \end{equation}
    of $K$-tori such that $Y$ admits a structure of a $K$-torsor under $R$ and the morphism $p$ is $\pi$-equivariant. Since 
    \[
        \pi_\ast[Y] = [T] = 0
    \] 
    in $\H^1(K,T)$, there exists $[\sigma] \in \H^1(K,S)$ that maps to $-[Y]$, or equivalently $R \cong Y^{\sigma}$. By Corollary~\ref{coro:DivA(T)Surjective}, there exists $r_\Abf \in R(\Abf_K)$ and $u \in T(K)$ such that $[r_\Abf] \in \Acal(R)_{\div}$ and 
    \[
        \pi(r_{\Abf}) u = t_{\Abf}.
    \]
    Let $\delta\colon T(K_v) \to \H^1(K_v,S)$ denote the connecting homomorphism, then
    \[
        [(p^{\sigma})^{-1}(t_v)] = [\pi^{-1}(t_v)] = \delta(t_v) = \delta(u)|_{K_v}
    \]
    in $\H^1(K_v,S)$ for every $v \in C^{(1)}$. Thus, $t_\Abf$ lifts to an adelic point on $Y^{\sigma + \tau } \simeq R^{\tau}$ where $\tau \in \operatorname{Z}^1(K,S)$ is a cocycle such that $[\tau] = -\delta(u)$. Therefore, $t_\Abf$ belongs to $T(\Abf_K)^{\tor}$ establishing part~\ref{prop:TorEqualTorTorForTori:1}. \par

    Let us prove~\ref{prop:TorEqualTorTorForTori:2}. Let $t_\Abf \in T(\Abf_K)^{\tor}$ and $p:Y \to T$ a torsor under a torus $S$. We might repeat the argument above to obtain an exact sequence like \eqref{eq extension tori} such that $Y$ is a $K$-torsor under $R$ together with $r_\Abf \in R(\Abf_K)$ and $u \in T(K)$ such that $[r_\Abf] \in \Acal(R)_{\div}$ and $\pi(r_{\Abf}) u = t_{\Abf}$. It is enough to conclude because $T(\Abf_K)^{\tor, \tor}$ is a subgroup of $T(\Abf_K)$ by Lemma~\ref{lemm:TorTorSubgroup}, and both $\pi(r_{\Abf})$ and $u$ belong to $T(\Abf)^{\tor, \tor}$.

\end{proof}

\noindent In order to deduce Theorem~\ref{thm:TorEqualTorTorTrivialPicardGroup} from Proposition~\ref{prop:TorEqualTorTorForTori}, we will need the following ``units torsor'' construction due to Skorobogatov.

\begin{lemm} \label{lemm:UnitTorsor}
    Let $X$ be a smooth geometrically integral variety over a field $k$. Then, the following are true:

    \begin{enumerate}
        \item \label{lemm:UnitTorsor:1} $\bar{k}[X]^\times / \bar{k}^\times$ is a finitely generated torsion-free abelian group.

        \item \label{lemm:UnitTorsor:2} Let $T$ be the $k$-torus with $\widehat{T} = \bar{k}[X]^\times / \bar{k}^\times$ and let $E_X$ be a $k$-torsor under $T$ whose class $[E_X] \in \H^1(k,T) \cong \Ext^1_k\bigl(\widehat{T},\bar{k}^\times\bigr)$ is the inverse class of the {\em units extension}
        \[
            1 \longrightarrow \bar{k}^\times \longrightarrow \bar{k}[X]^\times \longrightarrow \widehat{T} \longrightarrow 0
        \]
        given by Rosenlicht's lemma; {\em cf.} \cite[Proposition~3]{Rosenlicht1957Rationality}.
        Then, there exists a morphism
        \[
            u \colon X \longrightarrow E_X
        \]
        that induces an isomorphism of extensions
        \[
        \xymatrix{
            1 \ar[r] & \bar{k}^\times \ar@{=}[d] \ar[r] & \bar{k}[E_X]^\times \ar[d]^{u^\ast} \ar[r] & \widehat{T} \ar@{=}[d] \ar[r] & 0 \\ 
            1 \ar[r] & \bar{k}^\times \ar[r] & \bar{k}[X]^\times \ar[r] & \widehat{T} \ar[r] & 0.
        }
        \]

        We refer to $E_X$ as a {\em units torsor} and to $u$ as a {\em units morphism} for $X$.
        
        \item \label{lemm:UnitTorsor:3} The morphism $u$ from~\ref{lemm:UnitTorsor:2} is unique up to translation by a $k$-rational point $t \in T(k)$. 
    \end{enumerate}
\end{lemm}
\begin{proof}
    Claim~\ref{lemm:UnitTorsor:1} follows easily if $X$ can be embedded as a dense open subset of a normal proper variety $X'$, because in this case we have an inclusion 
    \[
        \bar{k}[X]^\times / \bar{k}^\times \longhookrightarrow \Div_{\infty}(\bar{X}'),
    \]
    where $\Div_{\infty}(\bar{X}')$ denotes the group of Weil divisors on $\bar{X'}$ supported outside $\bar{X}$. Claim~\ref{lemm:UnitTorsor:2} corresponds to \cite[Lemma~2.4.4]{Sko2001TorsorsAndRationalPoints}. 
    \par

    Let us prove~\ref{lemm:UnitTorsor:3}. Let $u_1$ and $u_2$ be two units morphisms $X \to E_X$. Note that the~morphism
    \[
        A:E_X \times T \longrightarrow E_X \times E_X, \qquad (e,t) \mapsto (e,t \cdot e)
    \]
    is an isomorphism. Then, we obtain a morphism
    \[
        \tau \colon X \longrightarrow T
    \]
    such that $u_2 = \tau \cdot u_1$ as $\operatorname{pr_2} \circ A^{-1}$ precomposed with $(u_1,u_2)\colon X \to E_X \times E_X$. We claim that $\tau$ is constant. Indeed, let $\chi\colon T \to \Gbf_{\m}$ be any character and let $f \in \bar{k}[E_X]^\times$ be an invertible function lifting $\chi$, i.e.
    \[
        f(t \cdot e) = \chi(t) \cdot f(e).
    \]
    Thus $u_2^\ast f = (\chi \circ \tau) \cdot (u_1^\ast f)$. But $u_1^\ast f$ and $u_2^\ast f$ have the same image in $\widehat{T}$ (namely $\chi$), hence
    \[
        \chi \circ \tau \colon X \longrightarrow \Gbf_{\m}
    \]
    is a constant for any character $\chi \in \widehat{T}$. Since these characters separate the geometric points of $T$, the morphism $\tau$ is constant. Its image is then a $k$-rational point $t \in T(\bar{k})$ such that $u_2(x) = t \cdot u_1(x)$ for all $x \in X$, which concludes the proof.  
\end{proof}

\noindent The following result is a key step in the proof Theorem~\ref{thm:TorEqualTorTorTrivialPicardGroup} that is also relevant for the proof of Theorem~\ref{thm:TorEqualTorTor}.

\begin{prop} \label{prop:UnitTorsorPreserveTor}
    Let $X$ be a smooth geometrically integral variety over $K$ with $\Pic(\bar{X}) = 0$. Let $T$ be the $K$-torus with $\widehat{T} = \bar{K}[X]^\times/\bar{K}^\times$ and let
    \[
        u\colon X \longrightarrow E_X
    \]
    be a units morphism for $X$; {\em cf.} Lemma~\ref{lemm:UnitTorsor}. Then, we have 
    \[
        u^{-1}(E_X(\Abf_K)^{\tor}) = X(\Abf_K)^{\tor}.
    \]
\end{prop}

\begin{proof}
     The contention $u^{-1}(E_X(\Abf_K)^{\tor}) \supseteq X(\Abf_K)^{\tor}$ is clear by functoriality of the toric descent set. Let $x_\Abf \in X(\Abf_K)$. Then, it is enough to show that $x_\Abf \in X(\Abf_K)^{\tor}$ whenever $u(x_\Abf) \in E_X(\Abf_K)^{\tor}$. Note that Lemma~\ref{lemm:UnitTorsor}~\ref{lemm:UnitTorsor:3} ensures that the property $u(x_\Abf) \in E_X(\Abf_K)^{\tor}$ is independent of the unit torsor $u$ chosen. Let
    \[
        Y \longrightarrow X
    \]
    be a torsor under a $K$-torus $T$. Let $R$ and $S$ be the $K$-tori with character lattices $\widehat{R} = \bar{K}[X]^\times/\bar{K}^\times$ and $\widehat{S} = \bar{K}[Y]^\times/\bar{K}^\times$ respectively. Since $\Pic(\bar{X}) = 0$, we have an exact~sequence
    \[
        1 \longrightarrow T \longrightarrow S \longrightarrow R \longrightarrow 1
    \]
    which is the Cartier dual of part of Sansuc’s exact sequence \cite[(6.10.1)]{Sansuc1981Arithmetique}. By virtue of \cite[Theorem~8.10]{HS2013Descent}, we can choose respective units morphisms $u_X$ and $u_Y$ for $Y$ and $X$ that fit in a catersian~square
    \begin{equation}\label{eq:UnitTorsorPreserveTor:1}
        \xymatrix{
            Y \ar[rr]^{u_Y} \ar[d] && E_Y \ar[d] \\
            X \ar[rr]^{u_X} && E_X,
        }
    \end{equation}
    where the right vertical arrow is a torsor under $T$ and where $u_Y$ is $T$-equivariant.
    \par
    
    Under the assumption $u_X(x_\Abf) \in E_X(\Abf_K)^{\tor}$, there exists $[\sigma] \in \H^1(K,T)$ such that $u_X(x_\Abf)$ lifts to an adelic point $e_\Abf \in E_Y^{\sigma}(\Abf_K)$. Twisting \eqref{eq:UnitTorsorPreserveTor:1} by $\sigma$ yields a catersian~square
    \[
        \xymatrix{
            Y^{\sigma} \ar[rr]^{u_{Y}^{\sigma}} \ar[d] && E_Y^\sigma \ar[d] \\
            X \ar[rr]^{u_X} && E_X.
        }
    \]
    The pair $(x_\Abf, e_\Abf)$ therefore defines an adelic point $y_\Abf \in Y^{\sigma}(\Abf_K)$ lifting $x_\Abf$. This shows that $x_\Abf \in X(\Abf_K)^{\tor}$ and concludes the proof.
\end{proof}

\begin{proof}[Proof of Theorem~\ref{thm:TorEqualTorTorTrivialPicardGroup}]
    In view of Theorem~\ref{thm:TorEqualMult}, we need to show that 
    \[
        X(\Abf_K)^{\tor} \subseteq X(\Abf_K)^{\tor,\tor}.
    \]
    We may assume $X(\Abf_K)^{\tor} \neq \varnothing$. Let $x_\Abf \in X(\Abf_K)^{\tor}$ and let consider a $K$-torsor
    \[
        Y \longrightarrow X
    \]
    under a $K$-torus $T$. We need to provide $[\sigma] \in \H^1(K,T)$ such that $x_\Abf$ lifts to $Y^{\sigma}(\Abf_K)^{\tor}$.
    \par

    Retake the notations $R$, $S$, $E_X$, $E_Y$, $u_X$, and $u_Y$ from the proof of Proposition~\ref{prop:UnitTorsorPreserveTor}. We have $u_X(x_\Abf) \in E_X(\Abf_K)^{\tor}$ by functoriality. By \cite[Theorem~6.4]{ILA2021LocalGlobal}, this implies in particular that $E_X(\Abf_K)^{\Be(E_X)} \neq \varnothing$. By \cite[Corollaire~8.3]{CTH2015DualCt} and \cite[Th\'eor\`eme~5.2]{Izquierdo2019Dualite}, one has $E_X(K) \neq \varnothing$, or $E_X \cong R$. It follows from Proposition~\ref{prop:TorEqualTorTorForTori} that $u_X(x_{\Abf}) \in E_X(\Abf_K)^{\tor,\tor}$. Therefore, the exists $[\sigma] \in \H^1(K,T)$ such that $u_X(x_{\Abf})$ lifts to an adelic point $e_\Abf \in E_Y^{\sigma}(\Abf_K)^{\tor}$. 
    \par
    
    Observe that twisting by $\sigma$ gives an isomorphism of extensions
    \[
        \xymatrix{
            1 \ar[r] & \bar{K}^\times \ar@{=}[d] \ar[r] & \bar{K}[Y^\sigma]^\times \ar[d]^{(u_Y^\sigma)^\ast} \ar[r] & \widehat{T} \ar@{=}[d] \ar[r] & 0 \\ 
            1 \ar[r] & \bar{K}^\times \ar[r] & \bar{K}[E_{Y}^\sigma]^\times \ar[r] & \widehat{T} \ar[r] & 0.
        }
    \]
    Here, the class in $ \Ext^{1}_K(\widehat{T},\bar{K}^\times)$ of the bottom row is exactly
    \[
        -[E_Y^\sigma] = -[E_Y] + \sigma
    \]
    It follows that $E_{Y}^\sigma$ is a units torsor and $u_Y^\sigma$ is a units morphism for $Y^{\sigma}$. Since $\Pic(\bar{X}) = \Pic(\bar{T}) = 0$, we have $\Pic(\bar{Y}^\sigma) = 0$ by Sansuc's exact sequence \cite[(6.10.1)]{Sansuc1981Arithmetique}. By the same argument from the last paragraph of the proof of Proposition~\ref{prop:UnitTorsorPreserveTor}, the pair $(x_\Abf, e_\Abf)$ defines an adelic point $y_\Abf \in Y^{\sigma}(\Abf_K)$ lifting $x_\Abf$. This implies that we have
    \[
        u_{Y}^{\sigma}(y_\Abf) = e_\Abf \in E^{\sigma}_Y(\Abf_K)^{\tor}.
    \]
     Now, the conclusion of Proposition~\ref{prop:UnitTorsorPreserveTor}, applied to $Y^\sigma$, shows that $y_\Abf \in Y^{\sigma}(\Abf_K)^{\tor}$. Thus, $x_\Abf$ belongs to $X(\Abf_K)^{\tor,\tor}$, which concludes the proof.
\end{proof}


\subsection{Varieties with finitely generated \texorpdfstring{$\Pic(\bar{X})$}{geometric Picard group}}

In this subsection, we establish Theorem~\ref{thm:TorEqualTorTor} in full generality. Fix a smooth geometrically integral variety $X$ over $K$ and assume $\Pic(\bar{X})$ finitely generated. Let us briefly sketch the argument: by Theorem~\ref{thm:TorEqualMult}, it remains to establish the inclusion
\[
    X(\Abf_K)^{\tor} \subseteq X(\Abf_K)^{\tor,\tor}.
\]
We need to show that any adelic point $x_\Abf \in X(\Abf_K)^{\tor}$ is also an element of $X(\Abf_K)^{\tor,\tor}$.

First, we find a torsor
\[
    Y \longrightarrow X
\]
under a quasitrivial torus $Q$, with $\Pic(\bar{Y}) = 0$. We then have a commutative diagram
\[
    \xymatrix{
        Y \ar[rr] \ar[d] && E_Y \ar[d] \\
        X \ar[rr] && E_X,
    }
\]
where the horizontal arrows are units morphisms; see Lemma~\ref{lemm:UnitTorsor}.
\par

Next, we follow the argument from Corollary~\ref{coro:QtBeNonempty}. Namely, we apply \cite[Theorem~6.4]{ILA2021LocalGlobal}, which exploits the nonvacuity of $X(\Abf_K)^{\tor}$, to show that $E_Y(\Abf_K)^{\Be(E_Y)} \neq \varnothing$. Since $E_Y$ is a $K$-torsor under a torus, this implies $E_Y(K) \neq \varnothing$ and {\em a fortiori} $E_X(K) \neq \varnothing$. The above diagram hence becomes
\[
    \xymatrix{
        Y \ar[rr] \ar[d] && S \ar[d] \\
        X \ar[rr] && T,
    }
\]
where the right vertical arrow is a morphism of tori.

We then take a lift $y_\Abf \in Y(\Abf_K)$ of $x_\Abf$ and proceed to find $q_\Abf \in Q(\Abf_K)$ such that
\[
    q_\Abf \cdot y_\Abf \in Y(\Abf_K)^{\tor}.
\]
Since $\Pic(\bar{Y}) = 0$, it then follows from Theorem~\ref{thm:TorEqualTorTorTrivialPicardGroup} that $q_\Abf \cdot y_\Abf \in Y(\Abf_K)^{\tor,\tor}$ and hence $x_\Abf \in X(\Abf_K)^{\tor,\tor}$, which then concludes the proof.
\par

To ensure the condition $q_\Abf \cdot y_\Abf \in Y(\Abf_K)^{\tor}$, it suffices to show that its image $s_\Abf \in S(\Abf_K)$ belongs to $S(\Abf_K)^{\tor}$, by virtue of Proposition~\ref{prop:UnitTorsorPreserveTor}. Then, Lemma~\ref{prop:TorEqualTorTorForTori} translates this into a completely arithmetic condition on the adelic space $\Acal(S) = S(\Abf_K)/S(K)$. To control this, we will need to cover both $Y$ and $S$ by a larger torsor $Z$ under a group of multiplicative type, and then lift $x_\Abf$ to a Galois twist of $Z$. This requires Theorem~\ref{thm:TorEqualMult}.

\medskip

\begin{lemm} \label{lemm:TorEqualTorTor:ExtentedTypeCoveringPicX}
    If $X(\Abf_K)^{\qt,\Be} \neq \varnothing$, then there exists a torsor
    \[
        Y \longrightarrow X
    \]
    under a quasitrivial $K$-torus $Q$ whose classical type
    \[
        \widehat{Q} \longrightarrow \Pic(\bar{X})
    \]
    is surjective. In particular, we have $\Pic(\bar{Y}) = 0$.
\end{lemm}

\begin{proof}
    Take a finitely generated subgroup of $\Div(\bar{X})$ that surjects onto $\Pic(\bar{X})$, and let 
    \[
        \Lcal \subseteq \Div(\bar{X})
    \]
    be its Galois orbit, which is again finitely generated by continuity of the Galois action.~Let
    \[
        \widehat{Q} \longtwoheadrightarrow \Lcal
    \]
    be a surjection from a permutation Galois module $\widehat{Q}$. This gives rise to a morphism of~complexes
    \[
        \lambda \colon \widehat{Q} \longrightarrow [\bar{K}(X)^\times/\bar{K}^\times \longrightarrow \Div(\bar{X})] \simeq \UPic(\bar{X})
    \]
    such that the composite $\widehat{Q} \xrightarrow{\lambda} \UPic(\bar{X}) \to \Pic(\bar{X})$ is surjective.
    \par
    
    Now, let $Q$ be the $K$-torus with character lattice $\widehat{Q}$. Since $X(\Abf_K)^{\qt,\Be} \neq \varnothing$, it follows from Corollary~\ref{coro:QtBeNonempty} that there exists a torsor $Y \to X$ of extended type $\lambda$. Sansuc's exact sequence \cite[(6.10.1)]{Sansuc1981Arithmetique} for this torsor reads
    \begin{equation} \label{eq:TorEqualTorTor:SansucSequence}
        0 \longrightarrow \bar{K}[X]^\times/\bar{K}^\times \longrightarrow \bar{K}[Y]^\times/\bar{K}^\times \longrightarrow \widehat{Q} \longrightarrow \Pic(\bar{X}) \longrightarrow \Pic(\bar{Y}) \longrightarrow \Pic(\bar{Q}) = 0
    \end{equation}
    Here, the morphism $\widehat{Q} \to \Pic(\bar{X})$ is the classical type of $Y$, that is, the composite
    \[
        \widehat{Q} \overset{\lambda}{\longrightarrow} \UPic(\bar{X}) \longrightarrow \Pic(\bar{X}).
    \]
    Therefore, this type is surjective by construction. Exactness of the sequence \eqref{eq:TorEqualTorTor:SansucSequence} shows that $\Pic(\bar{Y}) = 0$.
\end{proof}

\noindent Under the assumption $X(\Abf_K)^{\tor} \neq \varnothing$, it follows from \cite[Remark~6.3 and Theorem~6.4]{ILA2021LocalGlobal} that $X(\Abf_K)^{\qt,\Be} \neq \varnothing$. A torsor
\[
    Y \overset{Q}{\longrightarrow} X
\]
satisfying the condition of Lemma~\ref{lemm:TorEqualTorTor:ExtentedTypeCoveringPicX} therefore exists. Let $M$ be $K$-group of multiplicative type with $\widehat{M} = \Pic(\bar{X})$, and let $T$ and $S$ be $K$-tori with character lattices $\widehat{T} = \bar{K}[X]^\times/\bar{K}^\times$ and $\widehat{S} = \bar{K}[Y]^\times/\bar{K}^\times$ respectively. Sansuc’s exact sequence \eqref{eq:TorEqualTorTor:SansucSequence} for the torsor $Y \to X$~becomes
\[
    0 \longrightarrow \widehat{T} \longrightarrow \widehat{S} \longrightarrow \widehat{Q} \longrightarrow \widehat{M} \longrightarrow 0.
\]
Let $R \coloneq \Ker(S \to T) = \Img(Q \to S)$, which is a $K$-torus. Then, we have exact sequences
\[
    1 \longrightarrow R \longrightarrow S \overset{\tau}{\longrightarrow} T \longrightarrow 1 \qquad \text{and} \qquad  1 \longrightarrow M \longrightarrow Q \overset{\rho}{\longrightarrow} R \longrightarrow 1.
\]

\begin{lemm} \label{lemm:TorEqualTorTor:UnitTorsor}
    We can choose units morphisms for $X$ and $Y$ (see Lemma~\ref{lemm:UnitTorsor}) that  fit in a commutative square
    \[
    \xymatrix{
        Y \ar[rr]^{u} \ar[d] && E_Y \ar[d]^{\tau'} \\
        X \ar[rr] && E_X,
    }
    \]
    where $\tau'$ is a torsor under $R$ and where $u$ is $\rho$-equivariant.
\end{lemm}

\begin{proof}
    The pushforward $Y'\coloneq Y/M = Y \times^Q R$ is an $X$-torsor under $R$ whose type is~composite
    \[
        \widehat{R} \overset{\widehat{\rho}} \longrightarrow \widehat{Q} \longtwoheadrightarrow \Pic(\bar{X}) = \widehat{M},
    \]
    that is, the zero map. Sansuc's exact sequence \cite[(6.10.1)]{Sansuc1981Arithmetique} for the torsor $Y' \to X$~becomes
    \[
        0 \longrightarrow \widehat{T} \longrightarrow \bar{K}[Y']^\times/\bar{K}^\times \longrightarrow \widehat{R} \overset{0}{\longrightarrow} \Pic(\bar{X}),
    \]
    so $\bar{K}[Y']^\times/\bar{K}^\times$ is an extension of $\widehat{R}$ by $\widehat{T}$. By the five lemma, the pullback 
    \[
         \bar{K}[Y']^\times \longrightarrow \bar{K}[Y]^\times
    \]
    induces an isomorphism $\bar{K}[Y']^\times/\bar{K}^\times \cong \widehat{S}$. The same lemma tells us that this pullback is then an isomorphism, so for every unit morphism $Y' \to E_{Y'}$, the composite $Y \to Y' \to E_{Y'}$ is a unit morphism for $Y$.  By \cite[Theorem~8.10]{HS2013Descent}, we can choose units morphisms for $X$ and $Y'$ that fit in a cartesian square
    \[
    \xymatrix{
        Y' \ar[rr]^{u'} \ar[d] && E_{Y }\ar[d]^{\tau'} \\
        X \ar[rr] && E_X,
    }
    \]
    where $\tau'$ is a torsor under $R$ and where $u'$ is $R$-equivariant. Taking $u$ to be the composite
    \[
        Y \longtwoheadrightarrow Y' \overset{u'}{\longrightarrow} E_{Y}
    \]
    concludes the proof of the lemma.
\end{proof}

\begin{lemm} \label{lemm:TorEqualTorTor:UnitTorsorTrivial}
    If $X(\Abf_K)^{\qt,\Be} \neq \varnothing$, then $E_Y(K) \neq \varnothing$ and the morphism 
    \[
        \tau'\colon E_Y \longrightarrow E_X
    \]
    from Lemma~\ref{lemm:TorEqualTorTor:UnitTorsor} can be identify with the morphism $\tau\colon S \to T$ after appropriate trivializations of $E_Y$ and $E_X$.
\end{lemm}

\begin{proof}
    Take an arbitrary adelic point $x_\Abf \in X(\Abf_K)^{\qt,\Be}$. By definition, for any $\alpha \in \Be(E_Y)$, there exists a lifting $y_\Abf^{\alpha} \in Y(\Abf_K)$ that is orthogonal to $u^\ast \alpha \in \Be(Y)$. By functoriality, its image $u(y_\Abf^{\alpha}) \in E_Y(\Abf_K)$ is orthogonal to $\alpha$. Since $\alpha$ is locally constant, any element of $E_Y(\Abf_K) \neq \varnothing$ is orthogonal to $\alpha$. Furthermore, we have $E_Y(\Abf_K) \neq \varnothing$. Thus 
    \[
        E_Y(\Abf_K)^{\Be(E_Y)} = E_Y(\Abf_K) \neq \varnothing.
    \]
    Then $E_Y(K) \neq \varnothing$ by \cite[Corollaire~8.3]{CTH2015DualCt} and \cite[Th\'eor\`eme~5.2]{Izquierdo2019Dualite}. Choose a $K$-rational point $e \in E_Y(K)$ to trivialize $E_Y \cong S$, and then use $\tau'(e)$ to trivialize $E_X \cong T$. This makes the morphisms in the new diagram compatible with the toric action. The lemma is proved.
\end{proof}

\noindent Let $\pi$ be the composite $Q \overset{\rho}{\longtwoheadrightarrow} R \longhookrightarrow S$. By Lemma~\ref{lemm:TorEqualTorTor:KerRhoFiniteExponent} and Proposition~\ref{prop properties adelic spaces}~\ref{prop properties adelic spaces: finite exponent}, there exist integers $e_1,e_2 \ge 1$ such that
\begin{equation}\label{eq:TorEqualTorTor:TheIntegerE1}
    \Ker\left(\bar{\tau}\colon\frac{\Acal(S)}{\Acal(S)_{\div} + \pi(\Acal(Q))} \longrightarrow \frac{\Acal(T)}{\Acal(T)_{\div}} \right) 
\end{equation}
is $e_1$-torsion and
\begin{equation}\label{eq:TorEqualTorTor:TheIntegerE2}
    \left(\frac{\Acal(T)}{\Acal(T)_{\div}} \right)_{\tors}
\end{equation}
is $e_2$-torsion. \par
There exist surjective morphisms of tori
\[
    \chi_1\colon P_1 \longtwoheadrightarrow S \qquad \text{and} \qquad \chi_2\colon P_2 \longtwoheadrightarrow T
\]
where $P_1$ and $P_2$ are quasitrivial. Let 
\[
    P\coloneq Q \times P_1 \times P_2,
\]
which is a quasitrivial torus, and consider the morphism of tori
\[
    \Phi: P \longrightarrow S \times T, \qquad (q,p_1,p_2) \longmapsto \bigl(\pi(q) \chi_1(p_1)^{-e_1}, \tau(\chi_1(p_1)) \chi_2(p_2)^{-e_2} \bigr).
\]

\begin{lemm} \label{lemm:TorEqualTorTor:PhiSurjective}
    $\Phi$ is surjective.
\end{lemm}
\begin{proof}
    Consider elements $s \in S(\bar{K})$ and $t \in T(\bar{K})$. There exist $p_1 \in P_1(\bar{K})$ such that 
    \[
        \tau(\chi_1(p_1))^{e_1} = \tau(s)^{-1},
    \]
    because $\tau$, $\chi_1$, and $[e_1]_{S}$ are surjective. Since the kernel of $\tau$ is $R$, this means 
    \[
        \chi_1(p_1)^{e_1} s \in R(\bar{K}).
    \]
    By surjectivity of $\rho \colon Q \to R$, there exists $q \in Q(\bar{K})$ with
    \[
        \chi_1(p_1)^{e_1} s = \rho(q) = \pi(q).
    \]
    Finally, since $\chi_1$ and $[e_2]_{T}$  are surjective, there exists $p_2 \in P_2(\bar{K})$ such that
    \[
        \chi_2(p_2)^{e_2} = \tau(\chi_1(p_1)) t^{-1}.
    \]
    A direct calculation now shows that $\Phi(q,p_1,p_2) = (s,t)$.
\end{proof}

\begin{proof}[Proof of Theorem~\ref{thm:TorEqualTorTor}]
  We may assume $X(\Abf_K)^{\tor} \neq \varnothing$, which implies $X(\Abf_K)^{\qt,\Be} \neq \varnothing$ by \cite[Remark~6.3 and Theorem~6.4]{ILA2021LocalGlobal}. We keep all of the above notation. In particular, $Y\to X$ is a torsor under a quasitrivial torus whose classical type is surjective. By Lemma~\ref{lemm:TorEqualTorTor:PhiSurjective}, we have an exact sequence
\[
    1 \longrightarrow N \longrightarrow P \overset{\Phi}{\longrightarrow} S \times T \longrightarrow 1,
\]
where $N \coloneq \Ker(\Phi)$ is a $K$-group of multiplicative type. We construct a torsor 
\[
    Z \longrightarrow X
\]
under $N$ as follows: first, Lemma~\ref{lemm:TorEqualTorTor:UnitTorsorTrivial} ensures that we have a commutative diagram
\begin{equation} \label{eq:TorEqualTorTor:UnitTorsorTrivial}
    \xymatrix{
        Y \ar[rr]^{u} \ar[d] && S \ar[d]^{\tau} \\
        X \ar[rr] && T,
    }
\end{equation}
where $u$ is $\rho$-equivariant. Let $W \coloneq Y \times P_1 \times P_2$, which is an $X$-torsor under $P$. We define
\[
    \Psi\colon W \longrightarrow S \times T, \qquad (y,p_1,p_2) \longmapsto \bigl(u(y) \chi_1(p_1)^{-e_1}, \tau(\chi_1(p_1)) \chi_2(p_2)^{-e_2} \bigr),
\]
which is  $\Phi$-equivariant because $u$ is $\rho$-equivariant. The fibre
\[
    Z \coloneq \Psi^{-1}(1_S,1_T),
\]
is then an $X$-torsor under $\Phi^{-1}(1_S,1_T) = \Ker(\Phi) = N$.
\par

Let us now take an arbitrary adelic point $x_\Abf \in X(\Abf_K)^{\tor}$. By Proposition~\ref{prop:TorEqualMult}, there exists $[\sigma] \in \H^1(K,N)$ such that $x_\Abf$ lifts to an adelic point $z_\Abf \in Z^{\sigma}(\Abf_K)$. Since $\H^1(K,P) = 0$, there exists $(s,t) \in S(K) \times T(K)$ such that $[\Phi^{-1}(s,t)] = \delta(s,t) = -[\sigma]$, where
\[
    \delta\colon S(K) \times T(K) \longrightarrow \H^1(K,N)
\]
is the connecting homomorphism. Then
\[
    Z^\sigma \cong Z \times^N \Phi^{-1}(s,t) \cong \Psi^{-1}(s,t)
\]
as $X$-torsors under $N$. Under this identification, $z_\Abf$ corresponds to a point 
\[
    \bigl(y_\Abf, p_\Abf^{(1)}, p_\Abf^{(2)}) \in W(\Abf_K\bigr),
\]
where $y_\Abf \in Y(\Abf_K)$ lifts $x_\Abf$ and where $\bigl(p^{(1)}_\Abf,p^{(2)}_\Abf\bigr) \in P_1(\Abf_K) \times P_2(\Abf_K)$ satisfies
\[
    u(y_\Abf)\chi_1(p^1_\Abf)^{-e_1} = s \qquad \text{and} \qquad \tau(\chi_1(p^1_\Abf)) \chi_2(p^2_\Abf)^{-e_2} = t.
\]
\par 

Let $a,b \in \Acal(S)$ denote the respective classes of $u(y_\Abf)$ and $\chi_1(p^1_\Abf)$, and let $c \in \Acal(T)$ denote the class of $\chi_2(p^2_\Abf)$. The above equations implies that 
\[
    a = e_1b \qquad \text{and} \qquad \tau(b) = e_2c.
\]
Let us now look at diagram \eqref{eq:TorEqualTorTor:UnitTorsorTrivial}. Since $x_\Abf \in X(\Abf_K)^{\tor}$, its image in $T(\Abf_K)$, namely $\tau(u(y_\Abf))$, belongs to $T(\Abf_K)^{\tor}$ by functoriality of the toric descent set. By Proposition~\ref{prop:TorEqualTorTorForTori}, the class $\tau(a) \in \Acal(T)$ of $\tau(u(y_\Abf))$ is infinitely divisible. Thus
\[
    e_1e_2c = e_1 \tau(b) = \tau(a) \in \Acal(T)_{\div}.
\]
In particular, the image of $c$ in $\Acal(T)/\Acal(T)_{\div}$ is torsion. It follows that 
\[
    \tau(b) = e_2c \in \Acal(T)_{\div},
\]
by definition of $e_2$, {\em cf.} \eqref{eq:TorEqualTorTor:TheIntegerE2}. Then
\[
    a = e_1 b \in \Acal(S)_{\div} + \pi(\Acal(Q)),
\]
by definition of $e_1$, {\em cf.} \eqref{eq:TorEqualTorTor:TheIntegerE1}. This means there exists $q_\Abf \in Q(\Abf_K)$ such that
\[
    [\pi(q_\Abf) u(y_\Abf)] \in \Acal(S)_{\div}.
\]
By Proposition~\ref{prop:TorEqualTorTorForTori} and the $\pi$-equivariance of $u$, this implies
\[
    u(q_\Abf \cdot y_\Abf) = \pi(q_\Abf) u(y_\Abf) \in S(\Abf_K)^{\tor}.
\]
\par

\noindent Recall that $\Pic(\bar{Y}) = 0$ by Lemma~\ref{lemm:TorEqualTorTor:ExtentedTypeCoveringPicX}. By Proposition~\ref{prop:UnitTorsorPreserveTor} and Theorem~\ref{thm:TorEqualTorTorTrivialPicardGroup}, we have 
\[
    q_\Abf \cdot y_\Abf \in Y(\Abf_K)^{\tor} = Y(\Abf_K)^{\tor,\tor}.
\]
By functoriality of the iterated toric descent set, the adelic point $x_\Abf$, which is the image of $q_\Abf \cdot y_\Abf$, belongs to $X(\Abf_K)^{\tor,\tor}$. This concludes the proof of the theorem.
\end{proof}
 
\section{Applications to homogeneous spaces}\label{section Homogeneous spaces}

In this section, we obtain some direct applications of the results in Section~\ref{sec:TorsorOfGivenXType} to the local--global principle for homogeneous spaces over geometric global fields. Throughout this section, $K$ denotes a geometric global field and $C$ the associated Dedekind scheme.


\subsection{Preliminaries on homogeneous spaces}

Let $G$ be a linear algebraic group. We fix the following notation
\begin{itemize}
    \item $\mathrm{D}(G)$ is the derived subgroup of $G$.
    \item $G^{\circ}$ is the neutral connected component.
    \item $G^{\mathrm{f}} \coloneq G/ G^{\circ}$ is the group of components (a finite group).
    \item $G^{\mathrm{u}}$ denotes the unipotent radical (a characteristic unipotent subgroup).
    \item $G^{\mathrm{red}} \coloneq G / G^{\mathrm{u}}$ (a reductive group).
    \item $G^{\mathrm{ss}} \coloneq \operatorname{D}(G^{\mathrm{red}})$ (a semisimple group).
    \item $G^{\tor} \coloneq G^{\mathrm{red}}/ G^{\mathrm{ss}}$ (a torus).
    \item $G^{\mathrm{ssu}} \coloneq \ker \left( G^{\circ} \twoheadrightarrow G^{\tor}\right)$ (an extension of $G^{\mathrm{ss}}$ by $G^{\mathrm{u}}$).
    \item $G^{\mathrm{torf}} \coloneq G^{\circ}/ G^{\mathrm{ssu}}$ (an extension of $G^{\mathrm{f}}$ by $G^{\tor}$).
\end{itemize}

\begin{defn}\label{def condition Borovoi}
    Let $X$ be a homogeneous space of a connected linear group $G$ with connected geometric stabilisers $\bar{H}$. We say that $X$ satisfies the \textit{Borovoi condition} if $G^{\mathrm{ss}}$ is simply connected and $H^{\mathrm{torf}}$ is commutative (hence of multiplicative type).
\end{defn}

\noindent The local--global principle for homogeneous spaces satisfying the Borovoi condition was studied in \cite{ILA2021LocalGlobal}. There, Izquierdo and Lucchini Arteche introduced a set $\Omega_{X,0}$ of places depending on the homogeneous space $X$. The set $\Omega_{X,0}$ is obtained by adjoining finitely many places to $C^{(1)}$, but it is strictly larger than $C^{(1)}$ in general. We summarize the main results of \cite{ILA2021LocalGlobal} in the following proposition.

\begin{prop}{\cite[Theorem~3.2 and Corollary~5.6]{ILA2021LocalGlobal}}\label{prop ILA}
    Let $X$ be a homogeneous space of $G$ with geometric stabiliser $\bar{H}$. Assume that $X$ satisfies the Borovoi condition. Then, the following are true:
    \begin{enumerate}
        \item \label{prop ILA :1} If the natural map $\bar{H}^{\mathrm{torf}} \to \bar{G}^{\tor}$ is injective then the Brauer--Manin obstruction with respect to $\Be(X)$ is the only obstruction to the local--global for $X$.
        
        \item\label{prop ILA general homo} If $X(K_v)\neq \varnothing$ for all $v \in \Omega_{X,0}$ then
        \[
            X(\Abf_K)^{\qt, \Be} \neq \varnothing \Longrightarrow X(K)\neq\varnothing.
        \]
    \end{enumerate}
\end{prop}

\noindent In the rest of this section we explain how the results in Section~\ref{sec:TorsorOfGivenXType} allows us to  remove the assumption ``$X(K_v)\neq \varnothing$ for $v \in \Omega_{X,0}$'' in Proposition~\ref{prop ILA}~\ref{prop ILA general homo}.


\subsection{Springer class and descent types}

Let $X$ be a homogeneous space of an algebraic $K$-group $G$. Let $x \in X(\bar{K})$ and denote by $\bar{H}$ the stabiliser of $x$. A construction due to Springer associated to the pair $(X,x)$ a $K$-lien\footnote{or {\em band}, or {\em kernel}.} $L_X=(\bar{H}, \kappa_X)$ and a class $\eta_X \in \H^2(K,L_X)$ such that the class $\eta_X$ is neutral if and only if $X$ is dominated by a principal homogeneous space of $G$. Explicitly, the class $\eta_X$ is represented by the extension
\[
    1 \longrightarrow \bar{H}(\bar{K}) \longrightarrow E_X \longrightarrow \Gamma_K \longrightarrow 1,
\]
where $\Gamma_K$ is the absolute Galois group of $K$ and $E_X$ is the subgroup of $G(\bar{K}) \rtimes \Gamma_K$ that consists of elements $(g,\sigma)$ such that $\tensor[^\sigma]{x}{} = x \cdot g$. We denote by $\eta_X^{\mathrm{torf}}$ the image of $\eta_X$ in $\H^2(K, \bar{H}^{\mathrm{torf}})$. Note that since $\bar{H}^{\mathrm{torf}}$ is commutative, the class $\eta_X^{\mathrm{torf}}$ defines naturally a $K$-form $H^{\mathrm{torf}}$ of $\bar{H}^{\mathrm{torf}}$. When $\bar{H}$ is connected (i.e. when $H^{\mathrm{torf}}$ is a $K$-torus), we denote $\eta_X^{\mathrm{torf}}$ by $\eta_X^{\mathrm{tor}}$. See \cite[Section~2]{DLA2019HasseHomoReductionFinite} for more details on this.\par

It is known that period and index coincide for the fields under consideration, see \cite[Theorem~5.5]{HHK2029ApplicationsPatching} and \cite[Theorem~1.4]{CTGP2004ArithmeticLinear}. Then \cite[Proposition~5.5]{CTGP2004ArithmeticLinear} directly implies the following.

\begin{prop}\label{prop abelianisation springer class}
    Let $X$ be a homogeneous space of $G$ with connected geometric stabiliser. Then $\eta_X$ is neutral if and only if $\eta_X^{\mathrm{tor}} = 0$.
\end{prop}

Recently, the second named author introduced in \cite{ManhLinh2025DescentTypes} a general notion of \textit{descent type} of a torsor. Let $L = (\bar{H}, \kappa)$ be a $K$-lien. A descent type in $\operatorname{DType}(X,L)$ is represented by a pair $(\bar{Y}, E)$ where $\bar{Y} \to \bar{X}$ is a torsor under $\bar{H}$ and $E$ is an extension 
\[
    1 \longrightarrow \bar{H}(\bar{K}) \longrightarrow E \to \Gamma_K \longrightarrow 1,
\]
satisfying certain conditions.\par
For groups of multiplicative type, descent types and extended types are equivalent thanks to \cite[Theorem~4.1]{ManhLinh2025DescentTypes}. We rewrite part this comparison theorem below for the sake of the reader.

\begin{prop}\label{prop descent types vs extened types}
    Let $k$ be a field and $Z$ a smooth geometrically integral variety over $k$. Then, for every group of multiplicative type $M$, there exists a functorial morphism 
    \[
        \Phi \colon \operatorname{DType}(Z,M) \longrightarrow \Hom_{\mathcal{D}^+(k)}(\widehat{M}, \UPic(\bar{Z}))
    \]
    fitting in the following commutative diagram
    \begin{equation*}
        \begin{tikzcd}
            \operatorname{DType}(Z,M) \ar[d,"\Phi"] \ar[rd] & \\
            \Hom_{\mathcal{D}^+(k)}\bigl(\widehat{M}, \UPic(\bar{Z})\bigr) \ar[r] & \H^2(k,M),
        \end{tikzcd}
    \end{equation*}
    where the slanted arrow takes the class of each pair $(\bar{Y}, E)$ to the class $[E]$.
\end{prop}

\noindent The construction of the Springer class of $X$ and its pushforward $\eta_X^{\mathrm{torf}}$ gives rise to descent types as follows. The morphism $\bar{f_0} \colon \bar{G} \to \bar{X}$ given by $g \mapsto g \cdot x$ is a torsor under $\bar{H}$ and the pair $(\bar{f_0}, \eta_X)$ defines an element of $\operatorname{DType}(X,L_X)$. We make a slight abuse of notation here since the descent type is $[\bar{f_0}, E_X]$ for $E_X$ as above. Moreover, we can push $\bar{f_0}$ forward  to obtain the descent type in $\operatorname{DType}(X, L^{\mathrm{torf}})$ given by the torsor $\bar{f} \colon \bar{G} \times^{\bar{H}^{\mathrm{ssu}}} \bar{H}^{\mathrm{torf}} \to X$ and the class $\eta_X^{\mathrm{torf}}$.


\subsection{Homogeneous spaces with connected stabilisers}

Thanks to the tools of abelianization of nonabelian Galois cohomology, developed in \cite{Borovoi1993Abelianization}, and the dévissage techniques from \cite{BCTS08ElementaryObstruction}, we can easily deduce from the results of Section~\ref{sec:TorsorOfGivenXType} that the $(\qt, \Be)$-obstruction is the only obstruction to the existence of rational points on homogeneous spaces with connected stabilisers. Let us recall some basic facts about the elementary obstruction.\par

Let $X$ be a geometrically integral $K$-variety such that $\bar{k}[X]^{\times} = \bar{k}^{\times}$. The {\em elementary obstruction} $e(X) \in \operatorname{Ext}_K^2(\Pic(\bar{X}),\bar{k})$ is the class of the extension
\[
    1 \longrightarrow \bar{k}^{\times} \longrightarrow \bar{k}(X)^{\times} \longrightarrow \Div (\bar{X}) \longrightarrow \Pic (\bar{X}) \longrightarrow 0.
\]
The class $e(X)$ is an obstruction to the existence of rational points and it was first introduced and studied in \cite{CTS1987Descente}. When $\Pic(\bar{X})$ is finitely generated, the vanishing of $e(X)$ is equivalent to the existence of a universal torsor, see \cite[Section~2.3]{Sko2001TorsorsAndRationalPoints}. Then, we might combine Corollary~\ref{coro:QtBeNonempty} with \cite[Theorem~3.8]{BCTS08ElementaryObstruction} to obtain the following result:
\begin{prop}
    Let $X$ be a homogeneous space of a connected linear algebraic group $G$ with connected geometric stabiliser. Then $X(\Abf_K)^{\qt, \Be} \neq \varnothing$ implies that $X(K) \neq \varnothing$.
\end{prop}


\subsection{Homogeneous spaces satisfying the Borovoi condition}

In order to treat the case of a general homogeneous space of a linear group that satisfies the Borovoi condition, we need the following variant of \cite[Proposition~5.1]{ILA2021LocalGlobal}

\begin{prop}\label{prop existence of universal torsor that is a homogeneous space}
    Let $X$ be a homogeneous space of a connected linear algebraic  $K$-group $G$ withtorsor geometric stabilisers $\bar{H}$ that satisfy the Borovoi condition. Assume that $X(\Abf_K)^{\qt, \Be} \neq \varnothing$. Fix an embedding $H^{\mathrm{torf}} \hookrightarrow Q$ where $Q$ is a quasitrivial $K$-torus. Then there exists an $X$-torsor $Y \to X$ under $Q$ such that $Y$ is a homogeneous space of $G \times Q$ with geometric stabilisers $\bar{H}^{\mathrm{torf}}$. In particular, the natural map $\bar{H}^{\mathrm{torf}} \to (\bar{G}\times \bar{Q})^{\tor}$ is injective.
\end{prop}
\begin{proof}
    We can apply Corollary~\ref{coro:QtBeNonempty} to deduce that the ``connecting homomorphism''
    \[
        \partial \colon \Hom_{\Dcal^+(K)}\bigl(\widehat{Q},\UPic(\bar{X})\bigr) \longrightarrow \H^2(K,Q)
    \]
    is zero. On the other hand, we can consider the pushforward $[g, \xi] \in \operatorname{DType}(X, Q)$ of the descent type $[\bar{f}, \eta_X^{\mathrm{torf}}]$ described above. Since $\partial = 0$, Proposition~\ref{prop descent types vs extened types} ensures that $\xi = 0$ in $\H^2(K,Q)$. \par
    Let us translate the above conclusion to the language of gerbes. Triviality of $\xi$ implies that $\xi$ represents the trivial gerbe $\underline{\mathrm{TORS}}(Q)$. Additionally, the fact that $\xi$ is the pushforward of $\eta_X^{\mathrm{torf}}$ can be interpreted as a morphism $\mathcal{M}_X^{\mathrm{torf}} \to \underline{\mathrm{TORS}}(Q)$ of gerbes. Then, as in \cite[Proposition~5.1]{ILA2021LocalGlobal}, we conclude that this datum gives us the desired $X$-torsor under $Q$ by applying \cite[Théorème~3.4]{DLA2019HasseHomoReductionFinite}.
\end{proof}

\noindent We can now establish the following result.
\begin{thm}
    Let $X$ be a homogeneous space of $G$ satisfying the Borovoi condition. Then,
    \[
        X(\Abf_K)^{\qt, \Be} \neq \varnothing \Longrightarrow X(K) \neq \varnothing.
    \]
\end{thm}

\begin{proof}
    Fix an embedding $H^{\mathrm{torf}} \hookrightarrow Q$, where $Q$ is a quasitrivial torus. Let $f \colon Y \to X$ be the torsor under $Q$ from Proposition~\ref{prop existence of universal torsor that is a homogeneous space}. By assumption, $Y(\Abf_K)^{\Be(Y)}$ is nonempty and $Y$ satisfies the hypothesis of Proposition~\ref{prop ILA}~\ref{prop ILA :1}. We deduce that $Y(K) \neq \varnothing$ and thus $X(K)$ is nonempty.
\end{proof}

\medskip\par
{\em Acknowledgments.} We are grateful to Diego Izquierdo for bringing up this question. 
This work was initiated during the second author's postdoctoral appointment at the Institut de Math\'ematiques de Jussieu - Paris Rive Gauche, CNRS.

\bibliographystyle{alpha}
\bibliography{ref}

\end{document}